\documentclass[12pt]{article}
\usepackage{amssymb}
\usepackage{graphicx}
\usepackage{amsfonts}
\usepackage{graphicx}
\usepackage{amsmath}
\usepackage{amsthm}
\usepackage{amssymb}
\usepackage[colorlinks]{hyperref}  
\usepackage{color}

    \newtheorem{theorem}{Theorem}
    \newtheorem{lemma}[theorem]{Lemma}
    \newtheorem{proposition}[theorem]{Proposition}

\theoremstyle{definition} 

\newcommand{\eps}{\varepsilon}
\newcommand{\Z}{{\mathbb Z}}

\newcommand{\R}{{\mathbb R}}

\newcommand{\N}{{\mathbb N}}

\newcommand{\lstar}{{\raise-0.15ex\hbox{$\scriptstyle \ast$}}}

\theoremstyle{remark} 

\newcommand{\rr}{\mathbb{R}}
\newcommand{\ind}{{\bf 1}}

\newcommand{\re}{\textup{Re}}

\newcommand{\sch}{\textup{Sch}}
\newcommand{\Sineb}{\textup{Sine}_\beta}

\newcommand{\wt}{\widetilde}
\def\eqd{\stackrel{\scriptscriptstyle d}{=}}

\newcommand{\Pathschr}{\mathcal{J}_{\textup{Sch},T}}
\newcommand{\Pathsine}{\mathcal{J}_{\Sineb}}
\newcommand{\Psine}{\tilde{ \mathcal{J}}_{\textup{Sine}}}
\newcommand{\ldp}{\mathcal{I}}
\newcommand{\drift}{\mathfrak{f}}

\newcommand{\II}{\mathcal{H}}
\newcommand{\sineldp}{I_{\text{Sine}}}

\begin{document}

\title{Large deviations for the $\Sineb$ and $\sch_\tau$ processes}%

\author{Diane Holcomb\footnote{Department of Mathematics, University of Wisconsin - Madison,  holcomb@math.wisc.edu} \and Benedek Valk\'o\footnote{Department of Mathematics, University of Wisconsin - Madison, valko@math.wisc.edu}}%




\maketitle

\begin{abstract}
We study  two one-parameter families of point processes connected to random matrices: the $\Sineb$ and $\sch_\tau$ processes.  The first one is the bulk point process limit for the Gaussian $\beta$-ensemble. For $\beta=1, 2$ and $4$ it gives the limit of the GOE, GUE and GSE models of random matrix theory. In particular, for $\beta=2$ it is a determinantal point process conjectured to have similar behavior to  the critical zeros of the Riemann $\zeta$-function. The second process can be obtained as  the bulk scaling limit of the spectrum of certain discrete one-dimensional random Schr\"odinger operators. 

Both processes have asymptotically constant average density, in our normalization one expects close to $\tfrac{1}{2\pi}  \lambda$ points in a large interval of length $\lambda$. Our main results are large deviation principles for the average densities of the processes, essentially we compute  the asymptotic probability of seeing an unusual average density in a large interval. Our approach is based on the representation of the counting functions of these processes using stochastic differential equations. We also prove path level large deviation principles for the arising diffusions.

Our techniques work for the full range of parameter values. The results are novel even in the classical $\beta=1, 2$ and 4 cases for the $\Sineb$ process. They are consistent with the existing rigorous results on large gap probabilities and confirm the physical predictions made using log-gas arguments. 
%
%
\end{abstract}

\section{Introduction}

The Gaussian orthogonal, unitary and symplectic ensembles (GOE, GUE, GSE) are some of the most studied random matrix models.
These are symmetric (resp.~Hermitian or symplectic) matrices with i.i.d.~standard real (resp.~complex or quaternion) normal entries modulo the appropriate symmetry. It has been know from the classical results of Gaudin and Mehta \cite{mehta} that if we appropriately scale the spectrum in the bulk (e.g.~near zero) then we obtain a  limiting point process. The point process can be described via its $n$-point correlation functions. These are given by determinantal formulas in the GUE case and Pfaffian formulas in the GOE, GSE cases. (See  \cite{AGZ}, \cite{ForBook}, \cite{mehta} for more details and the precise description.)

The GOE, GUE, GSE models can be naturally included in a one-parameter family of distributions. The joint eigenvalue distribution for these classical models is known to be
\begin{align}\label{betaens}
p(\lambda_1,\dots,\lambda_n)=\frac{1}{Z_{n,\beta}} \prod_{1\le i<j\le n} |\lambda_i-\lambda_j|^\beta e^{-\frac{\beta}4 \sum\limits_{i=1}^n x^2}
\end{align}
where $\beta$ is equal to 1, 2 and 4 in the three cases. Note, that the constant $\beta/4$  in the exponential can be easily changed via linear scaling.  It is natural to consider the density (\ref{betaens}) for general $\beta>0$, this is the Gaussian (or Hermitian) $\beta$-ensemble. In \cite{BVBV} the authors show the existence of the bulk scaling limit for general $\beta$. In particular, if $\Lambda_{n,\beta}$ is distributed according to (\ref{betaens}) then $2 \sqrt{n} \Lambda_{n,\beta}$ converges to a random point process, denoted by $\Sineb$. For $\beta=1,2, 4$ this gives the bulk limit process for the GOE, GUE, GSE ensembles.

The $\Sineb$ process can be described through its counting function using a system of stochastic differential equations.  Consider the system
\begin{align}\label{sineSDE1}
d \alpha_\lambda=\lambda \frac{\beta}{4}e^{-\tfrac{\beta}{4} t} dt+\re\left[(e^{-i \alpha_\lambda}-1)(dB_1+i dB_2)  \right], \qquad \alpha_\lambda(0)=0, \quad t\in [0,\infty)
\end{align}
where $B_1, B_2$ are independent standard Brownian motions. Note, that this is a one-parameter family of SDEs driven by the same complex Brownian motion. In \cite{BVBV} it was shown that $N_\beta(\lambda)=\lim\limits_{t\to \infty} \frac{1}{2\pi} \alpha_\lambda(t)$ exists almost surely and it is an integer valued monotone increasing function in $\lambda$. Moreover, the function $\lambda\to N_\beta(\lambda)$ has the same distribution as the counting function of the $\Sineb$ process, i.e.~the distribution of the number of points in $[0,\lambda]$ for $\lambda>0$ is given by that of $N_\beta(\lambda)$. 

Note, that for any fixed $\lambda$ the process $\alpha_\lambda$ satisfies the SDE
\begin{align}\label{sineSDE2}
d \alpha_\lambda=\lambda \frac{\beta}{4}e^{-\tfrac{\beta}{4} t} dt+2 \sin(\alpha_\lambda/2) dB_t, \qquad \alpha_\lambda(0)=0, \quad t\in [0,\infty)
\end{align}
where $B_t=B_t^{(\lambda)}=\int_0^t \re\big[-2 i e^{-\tfrac{1}{2}{i \alpha_\lambda(s)}}d(B_1+i dB_2)\big]$ is a standard Brownian motion which depends on $\lambda$. Thus, if we are interested in the number of points in a given interval $[0,\lambda]$ then it is enough to study the SDE (\ref{sineSDE2}) instead of the system (\ref{sineSDE1}).

Using the SDE characterization of the $\Sineb$ process one can show that it is  translation invariant with density $(2\pi)^{-1}$ (see \cite{BVBV}).  In particular, in a large interval $[0,\lambda]$ one expects roughly $(2\pi)^{-1} \lambda$ points.  In  \cite{KVV} the authors refined this by showing that $N_\beta(\lambda)$ satisfies a central limit theorem, it is asymptotically normal with mean $\frac{\lambda}{2\pi}$ and variance $\frac{2}{\beta \pi^2} \log \lambda$.

The goal of the current paper is to characterize the large deviation behavior of $N_\beta(\lambda)$.  We will find the asymptotic probability of seeing an average density different from $(2\pi)^{-1}$ on a large interval. Our main theorem will show that $\lambda^{-1} N_\beta(\lambda)$ satisfies a large deviation principle with a  good rate function.

Before stating the exact form of the theorem we need to introduce a couple of notations.  We will use
\begin{align}
K(a)= \int_0^{\pi/2} \frac{dx}{\sqrt{1-a \sin^2 x}}, \hspace{1cm}
E(a)= \int_0^{\pi/2} \sqrt{1-a\sin^2 x}dx, \label{KandE}
\end{align}
for the complete elliptic integrals of the first and second kind, respectively. Note that there are several conventions denoting these functions, we use the one in \cite{AbSt}.
 We also introduce the following  function for $a<1$:
\begin{align}
\label{defH}
\II(a)&= (1-a)K(a)-E(a).
\end{align}
Now we are ready to state our main theorem.
%
\begin{theorem}\label{thm:sineb1}
Fix $\beta>0$. The sequence of random variables $\frac{1}{ \lambda}N_\beta(\lambda)$ satisfies a large deviation principle with scale $\lambda^2$ and good rate function $\beta \sineldp (\rho)$ with
\begin{align}\label{defsineldp}
\sineldp (\rho)= \frac{1}{8} \left[ \frac{\nu}{8}+ \rho \II(\nu)\right],\qquad \nu=\gamma^{(-1)}(\rho),
\end{align}
where $\gamma$ is a continuous, strictly decreasing function given by 
\begin{align}
\gamma(\nu)= \begin{cases} \,\,\frac{\II(\nu)}{8} \int\limits_{-\infty}^\nu \II^{-2}(x)dx, & \textup{ if }\nu<0,\\[14pt]
\qquad\quad  \tfrac{1}{2\pi}, & \textup{ if } \nu=0,\\[14pt]
\,\, \frac{\II(\nu)}{8} \int\limits_{1}^\nu \II^{-2}(x)dx, & \textup{ if } 0< \nu< 1,\\[14pt]
\qquad\quad 0, &\textup{ if } \nu=1.
 \end{cases}\label{rhonu}
\end{align}
\end{theorem}
Roughly speaking, this means that the probability of seeing close to  $\rho \lambda$ points in $[0,\lambda]$ for a large $\lambda$ is asymptotically $e^{-\lambda^2 \beta \sineldp (\rho)}$. The precise statement is that if $G$ is an open, and $F$ is a closed subset of $[0,\infty)$ then
\begin{align*}
\liminf_{\lambda\to \infty} \frac{1}{\lambda^2} P(\tfrac{1}{\lambda}N_\beta(\lambda)\in G)&\ge -\inf_{x\in G} \beta \sineldp (x), \quad
\limsup_{\lambda\to \infty} \frac{1}{\lambda^2} P(\tfrac{1}{\lambda}N_\beta(\lambda)\in F)\le -\inf_{x\in F} \beta \sineldp (x).
\end{align*}

The function $\gamma$ may also be defined as the solution to the  equation $4x(1-x)\gamma''(x)=  \gamma(x)$ on the intervals $(-\infty,0]$ and $[0,1]$ with boundary conditions $\lim\limits_{x\to 0^{\pm}} \gamma(x)=\tfrac{1}{2\pi}$, $\gamma(1)=0$ and $\lim\limits_{x\to -\infty} \frac{\gamma(x)}{\sqrt{|x|}}=\tfrac14$. 
The rate function $\sineldp (\rho)$ is strictly convex and  non-negative with $\sineldp(\tfrac1{2\pi})=0$ and $\sineldp(0)=\tfrac1{64}$. The function $\sineldp (\tfrac1{2\pi}+x)$ behaves like $-\tfrac{\pi^2 x^2}{4 \log (1/|x|)}$ for small $|x|$, and $\sineldp(\rho)$ grows like $\frac{1}{2} \rho^2 \log\rho$ as $\rho\to \infty$.  These statements will be proved in Proposition \ref{prop:sineldp}.

We  note that the behavior of $\sineldp(\rho)$ near $\rho=\tfrac1{2\pi}$ is formally consistent with the already mentioned  central limit theorem of $N_\beta(\lambda)$. For $\rho=\tfrac{1}{2\pi}+x$ with a small, but fixed $|x|$ the probability of seeing close to $\tfrac{1}{2\pi}\lambda+x \lambda$ points in $[0,\lambda]$ is approximately $\exp\big(\!\!-\tfrac{\beta \pi^2 \lambda^2 x^2}{4 \log(1/|x|)}\big)$.  Now let us assume, that this is true even if $x$ decays with $\lambda$, even though this regime is not covered in our theorem. If we substitute $\lambda x=\sqrt{\tfrac{2}{\beta \pi^2} \log \lambda} \cdot y$ (with a fixed $y$), then this  probability would asymptotically equal to $e^{-{y^2}/{2}}$. This is in agreement with the fact that $N_\beta(\lambda)$ is asymptotically normal with mean $\tfrac{1}{2\pi}\lambda$ and variance $\tfrac{2}{\beta \pi^2} \log \lambda$. 

Before moving on,  a couple of historical notes are in order. In \cite{BVBV} the authors also show another large deviation statement for the $\Sineb$ process regarding large intervals, namely that the asymptotic probability of not seeing any points in $[0,\lambda]$ is approximately $e^{-\tfrac{\beta}{64} \lambda^2}$. In \cite{BVBV2} this result was sharpened by providing the more precise asymptotics of 
\begin{align}\label{large_gap}
P(N_\beta(\lambda)=0)=(\kappa_\beta+o(1)) \lambda^{\upsilon_\beta} \exp\left\{-\tfrac{\beta}{64}\lambda^2+\left(\tfrac{\beta}{8}-\tfrac14\right)\lambda\right\}, \qquad \textup{as $\lambda\to \infty$}
\end{align}
with $\upsilon_\beta=\tfrac{1}{4}\left(\tfrac{\beta}{2}-\tfrac{2}{\beta}-3\right)$ and a positive constant $\kappa_\beta$ whose value was not determined. Similar results have been proven before for the classical cases $\beta=1, 2, 4$, see e.g.  \cite{BTW}, \cite{TW93}, \cite{Wi96}, \cite{DIZ96}. Moreover, the value of $\kappa_\beta$ and higher order asymptotics were also established for these specific cases by \cite{Kr04}, \cite{Ehr06}, \cite{DIKZ07}.  Further extension in the classical cases include the exact asymptotics of $P(N_\beta(\lambda)=n)$ for fixed $n$ and also for $n=o(\lambda)$. (See \cite{TW93} and \cite{ForBook} for details.) In all of  these results the main term of the asymptotic probability is  $e^{-\tfrac{\beta}{64} \lambda^2}$. This is consistent with our result, as Theorem \ref{thm:sineb1} and $ \sineldp (0)=\frac{1}{64}$ implies
\begin{align*}
\lim\limits_{\eps\to 0} \lim\limits_{\lambda\to \infty} \tfrac{1}{\lambda^2} \log P(N_\beta(\lambda)\le \eps \lambda)=-\tfrac{\beta}{64}.
\end{align*}
The large deviation rate function (\ref{defsineldp}) has been predicted using non-rigorous scaling and log-gas arguments in \cite{Dyson95} and \cite{FS95}. (See Section 14.6 of \cite{ForBook} for an overview.) Using the same techniques \cite{FW} treats the corresponding problem for the soft edge and hard edge limit processes of $\beta$-ensembles.


 One can also study the large deviation behavior of the empirical distribution of the $\beta$-ensembles on a macroscopic level. 
It is known that after scaling with $\sqrt{n}$ the empirical measure of the distribution (\ref{betaens}) converges to the Wigner semicircle law. In \cite{BAG} the authors prove a large deviation principle for the scaled empirical measure, this describes the asymptotic probability of seeing a different density profile than the semicircle. One could consider our theorem a microscopic analogue of that result. 
 \medskip

We will also consider another, related symmetric random matrix ensemble. Let $H_{n, \sigma}$ be a random symmetric tridiagonal matrix with entries equal to 1 above and below the diagonal and i.i.d.~normals with mean zero and variance $\tfrac{\sigma^2}{n}$ on the diagonal.  
\begin{equation}\label{shrod1dmatrix}
H_{n,\sigma}=\left( \begin{array}{ccccc}
\omega_1 & 1 &  &   &\\
1 & \omega_2& 1    && \\
  & 1  &\ddots & &\\
  & & &\ddots &1 \\
 & & &1 & \omega_n \\
\end{array} \right), \qquad \omega_i\sim N(0,\sigma^2 n^{-1} ).
\end{equation}
The matrix $H_{n,\sigma}$ can be viewed as a one-dimensional discrete random Schr\"odinger operator. In \cite{KVV} it was shown that the bulk scaling limit of the spectrum of  $H_{n,\sigma}$ (along appropriate subsequences) is given by a one parameter family of point processes with density $(2\pi)^{-1}$ denoted by $\sch_\tau$. (The parameter $\tau>0$ depends on $\sigma$ and the point in the spectrum where we zoom in to take the limit.) The process $\sch_\tau$ can be characterized via its counting function in a similar way to the $\Sineb$ process. Consider the following one-parameter family of SDEs:
\begin{align}\label{schrSDE1}
d \phi_\lambda=\lambda dt+dB_0+\re\left[e^{-i \phi\lambda} (dB_1+i dB_2)  \right], \qquad \phi_\lambda(0)=0, \quad t\in [0,\infty)
\end{align}
where $B_0, B_1, B_2$ are independent standard Brownian motions. Then the random set
\begin{align*}
\Lambda_\tau:=\{\lambda: \phi_{\lambda/\tau}(\tau)\in 2\pi \Z\}
\end{align*}
has the same distribution as $\sch_\tau$. 
Denote the counting function of the process by $\wt N_\tau$, i.e.~for $\lambda>0$ let $\wt  N_\tau(\lambda)=\#(\sch_\tau\cap [0,\lambda])$. 
In \cite{KVV} it was shown that $\wt  N_\tau(\lambda)$ is close to a normal with mean $\tfrac{\lambda}{2\pi}$ and  a constant variance $\tfrac{\tau}{4\pi^2}$. In our next result we derive the large deviation behavior of  $\wt  N_\tau(\lambda)$, this is the analogue of Theorem \ref{thm:sineb1} for the $\sch_\tau$ processes.


\begin{theorem}\label{thm:schr1}
Fix $\tau>0$. The sequence of random variables $\frac{1}{ \lambda}\wt N_\tau(\lambda)$ satisfies a large deviation principle with scale $\lambda^2$ and rate function $\frac{1}{\tau} I_{\sch}(\cdot)$ where  $I_{\sch}(\rho)=\ldp(2\pi \rho)$ and for $q>0$
\begin{align}
\ldp(q)&
=\frac{2-a}{8}-\frac{E(a)}{4 K(a)}, \quad \textup{with }\quad a=a(q)=K^{-1}(\pi/(2q)).\label{ldpdef}
\end{align}
\end{theorem}
The rate function $ I_{\sch}(\rho)$ is strictly convex and locally quadratic at the absolute minimum point $\rho=\tfrac{1}{2\pi}$.  (See Proposition \ref{prop:ldpconvex}.) The local behavior of $ I_{\sch}(\rho)$ at $\rho=\tfrac{1}{2\pi}$ is formally consistent with the fact that $N_\tau(\lambda)-\tfrac{\lambda}{2\pi}$ is close to a normal random variable with a constant variance $\tfrac{\tau}{4\pi^2}$.


The proofs of Theorems \ref{thm:sineb1} and \ref{thm:schr1} will rely on path level large deviation principles on the corresponding stochastic differential equations. These in turn will follow by analyzing the hitting time of $2\pi$ for the diffusion
\begin{align}\label{schrSDE0}
d\wt \alpha_\lambda=\lambda dt+2 \sin(\wt \alpha_\lambda/2) dB, \qquad \wt \alpha_\lambda(0)=0, \quad t\in [0,\infty).
\end{align}
Note, that  for a  fixed $\lambda$ the process $\wt \alpha_\lambda(t)$ is equal in distribution to $\phi_\lambda(t)-\phi_0(t)$ from (\ref{schrSDE1}).

In the next section we summarize some of the important properties of the SDEs we work with, and state the needed path level large deviation results. In Section \ref{s:hitting} we study  diffusion $\wt \alpha_\lambda$ of  (\ref{schrSDE0}) using the Cameron-Martin-Girsanov change of measure technique. In Sections \ref{s:pathschr} and \ref{s:pathsine} we derive  path level large deviations for the diffusions $\alpha_\lambda$ and $\wt \alpha_\lambda$ from (\ref{sineSDE2}) and (\ref{schrSDE0}).
In Section \ref{s:goodrate} we analyze the rate functions for the path level large deviations and in Section \ref{s:endpoint} we complete the proofs of Theorems \ref{thm:sineb1} and \ref{thm:schr1}. In the Appendix we will discuss various properties and asymptotics of the used special functions.

\section{Properties of the diffusions corresponding to $\Sineb$ and $\sch_\tau$}\label{s:diff}

Our starting point is the observation that if $\lambda>0$ is fixed, then  if the diffusion $\wt \alpha_\lambda$ (defined in (\ref{schrSDE0})) hits $2n \pi$ for $n\in \Z$, it will stay above it. This can be seen from the fact that when $\wt \alpha_\lambda$ hits $2n \pi$ the noise term vanishes, but the drift term is always positive.
Introduce the notations
\begin{align}\notag
\lfloor y \rfloor_{2\pi}= \max\{ 2\pi k: 2\pi k \leq y\}, \qquad  \lceil y \rceil_{2\pi}= \min\{ 2\pi k: 2\pi k \geq y\}.
\end{align}
 From the strong Markov property we immediately get the following proposition.
\begin{proposition}\label{prop:renewal}
Fix $\lambda>0$. Then the process $\lfloor \wt  \alpha_\lambda(t) \rfloor_{2\pi}$ is non-decreasing in $t$. Moreover, the waiting times between the jump times of this process are i.i.d.~with the same  distribution as the hitting time 
\begin{align}\label{hitting}
\tau_\lambda=\inf \{t: \wt \alpha_\lambda(t)\ge 2\pi\}.
\end{align}
\end{proposition}
Consider the diffusions $\wt \alpha^{(1)}_\lambda$ and $\wt \alpha^{(2)}_\lambda$ which are strong solutions of the SDE (\ref{schrSDE0}), but with initial conditions $\wt \alpha^{(1)}_\lambda(0)=c_1\le \wt \alpha^{(2)}_\lambda=c_2$. Then a simple coupling argument shows that $\wt \alpha^{(1)}_\lambda(t) \le \wt \alpha^{(2)}_\lambda(t)$ for all $t\ge 0$. Our next proposition will build on this statement using the strong Markov property. 
\begin{proposition}\label{prop:stochdom}
Let $0=t_0<t_1<t_2<\dots<t_n=T$ and fix a $\lambda>0$. Consider the solution $\wt \alpha_\lambda(t)$ of (\ref{schrSDE0}) on $[0,T]$. Then there exists independent random variables $\xi_1, \xi_2, \dots, \xi_n$ so that 
\begin{align}\label{order}
 \lfloor\xi_i \rfloor_{2\pi} \le \lfloor \wt \alpha_\lambda(t_i)\rfloor_{2\pi}-\lfloor \wt \alpha_\lambda(t_{i-1})\rfloor_{2\pi}\le  \lfloor\xi_i \rfloor_{2\pi}+2\pi, \quad 1\le i \le n,
\end{align}
and $\xi_i$ is distributed as $\wt \alpha_\lambda(t_i-t_{i-1})$.
\end{proposition}
\begin{proof}
Let $\hat \alpha_i(s)$ be defined as the strong solution of (\ref{schrSDE0}) on $[t_{i-1},t_i]$ with initial condition $\hat \alpha_i(t_{i-1})=0$ and let $\xi_i= \hat \alpha_i(t_i)$. Clearly,
$\xi_i, 1\le i \le n$ are independent random variables and $\xi_i\eqd \wt \alpha_\lambda(t_i-t_{i-1})$, we just have to show that (\ref{order}) holds. Fix an integer $1\le i \le n$ and define 
\begin{align*}
\wt \alpha^{(1)}_\lambda(s)=\hat \alpha_i(s)+\lfloor \wt \alpha_\lambda(t_{i-1})\rfloor_{2\pi}, \qquad \wt \alpha^{(2)}_\lambda(s)=\hat \alpha_i(s)+\lfloor \wt \alpha_\lambda(t_{i-1})\rfloor_{2\pi}+2\pi, \qquad s\in [t_{i-1},t_i].
\end{align*}
Then $\wt \alpha_\lambda, \wt \alpha^{(1)}_\lambda, \wt \alpha^{(2)}_\lambda$ are all strong solutions of (\ref{schrSDE0}) on $ [t_{i-1},t_i]$ with initial conditions
\begin{align*}
\wt \alpha^{(1)}_\lambda(t_{i-1})\le \wt \alpha_\lambda(t_{i-1})\le  \wt \alpha^{(2)}_\lambda(t_{i-1})= \wt \alpha^{(1)}_\lambda(t_{i-1})+2\pi.
\end{align*}
\begin{figure}[!t]
\centering
\includegraphics[width=0.8\textwidth]{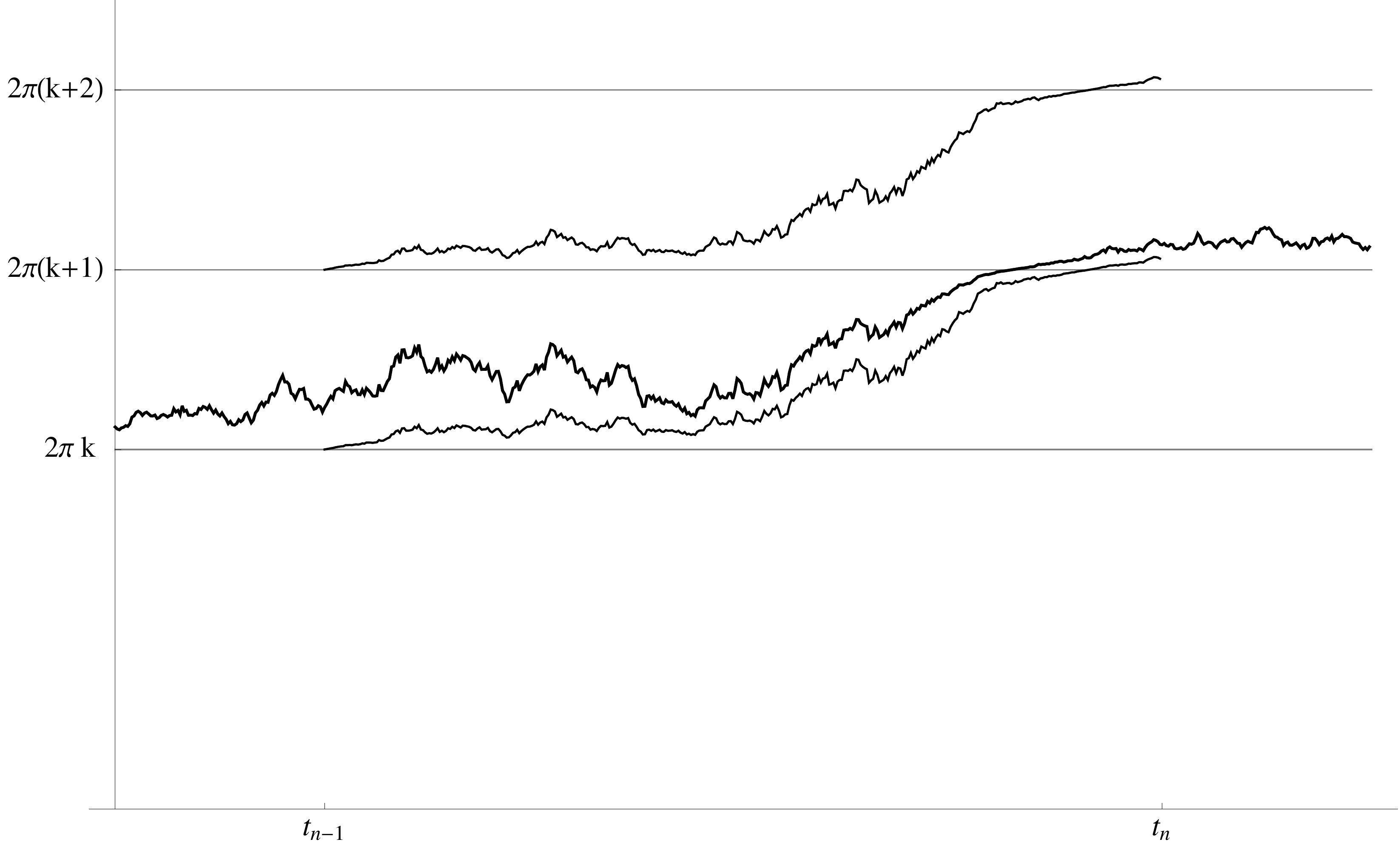}
\caption{ \footnotesize{The coupling of Proposition \ref{prop:stochdom}. The process $\wt \alpha_\lambda$ is the diffusion in the middle, it is sandwiched between $\wt \alpha^{(1)}_\lambda$ and $\wt \alpha^{(2)}_\lambda=\wt \alpha^{(1)}_\lambda+2\pi$ which start at integer multiples of $2\pi$ at the beginning of the coupling interval.} }\label{fig:coupling}
\end{figure}
The ordering is preserved by the coupling so we have
\begin{align*}
\wt \alpha^{(1)}_\lambda(t_{i})\le \wt \alpha_\lambda(t_{i})\le  \wt \alpha^{(2)}_\lambda(t_{i})= \wt \alpha^{(1)}_\lambda(t_{i})+2\pi.
\end{align*}
(See Figure \ref{fig:coupling} for an illustration.) From this we get 
\begin{align*}
\lfloor \wt \alpha_\lambda(t_i)\rfloor_{2\pi}-\lfloor \wt \alpha_\lambda(t_{i-1})\rfloor_{2\pi}&=
\lfloor \wt \alpha_\lambda(t_i)\rfloor_{2\pi}-\lfloor \wt \alpha^{(1)}_\lambda(t_{i-1})\rfloor_{2\pi}\\
&\ge \lfloor \wt \alpha^{(1)}_\lambda(t_i)\rfloor_{2\pi}-\lfloor \wt \alpha^{(1)}_\lambda(t_{i-1})\rfloor_{2\pi}= \lfloor\xi_i \rfloor_{2\pi},
\end{align*}
and
\begin{align*}
\lfloor \wt \alpha_\lambda(t_i)\rfloor_{2\pi}-\lfloor \wt \alpha_\lambda(t_{i-1})\rfloor_{2\pi}&=
\lfloor \wt \alpha_\lambda(t_i)\rfloor_{2\pi}-\lfloor \wt \alpha^{(1)}_\lambda(t_{i-1})\rfloor_{2\pi}\\
&\le  \lfloor \wt \alpha^{(2)}_\lambda(t_i)\rfloor_{2\pi}-\lfloor \wt \alpha^{(1)}_\lambda(t_{i-1})\rfloor_{2\pi}= \lfloor\xi_i \rfloor_{2\pi}+2\pi.\qedhere
\end{align*}
\end{proof}
We will also need another type of coupling for a slightly more general family of diffusions. 
Consider the SDE
\begin{align}\label{xi}
d\xi_{f,c}=f dt+\re ((e^{-i \xi_{f,c}}-1)(dB_1+i dB_2)), \quad \xi_{f,c}(0)=c, \quad t\in [0,\infty)
\end{align}
where $f$ is an integrable non-negative function. Note, that for fixed $f, c$ this process has the same distribution as
\begin{align}\label{xisine2}
d\wt \xi_{f,c}=f dt+2 \sin(\wt \xi_{f,c}/2) dB, \quad \wt \xi_{f,c}(0)=c, \quad t\in [0,\infty)
\end{align}
The following properties of $\xi_{f,c}$ follow from the basic theory of diffusions and standard coupling arguments. 
\begin{proposition}\label{prop:SDEprop}
\begin{enumerate}
\item[(i)] Let $\tau_{2\pi n}$ be the hitting time of $2\pi n$, where $2\pi n>c$ and $n$ is an integer. Then for any $t>\tau_{2\pi n}$ we have $\xi_{f,c}\ge 2\pi n$. In particular, if $c\ge 0$ then $\xi_{f,c}(t)$ stays non-negative for all $t>0$. 

\item[(ii)]    If $f\ge g$ and $\xi_{f,a}$ and $\xi_{g,b}$ are driven by the same Brownian motions then $\xi_{f,a}-\xi_{g,b}$ has the same distribution as $\xi_{f-g,a-b}$. If $a\ge b$ then $\xi_{f,a}-\xi_{g,b}$ stays a.s.~non-negative for all $t$.  

%
%
%

\item[(iii)]  For any finite $T$ we have the following exponential tail bound
\begin{align}\label{expbnd}
P(\xi_{f,0}(T)\ge ka)\le 2  \left(\frac{\int_0^T f(t) dt}{2\pi a}\right)^k, \qquad k\in \N.
\end{align}
 If $\int_0^\infty f(t) dt<\infty$ then $\xi_{f,c}(\infty)=\lim\limits_{t\to \infty} \xi_{f,c}(t)$ exists a.s.~and  the previous bound holds for $T=\infty$ as well.

 \end{enumerate}
\end{proposition}
\begin{proof}[Sketch of the proof] The first statement follows  from the strong Markov property and the fact that in (\ref{xisine2}) the noise term vanishes if $\wt \xi_{f,c}\in 2\pi \Z$, but the drift is always non-negative. The first part of  (ii) follows
 by considering the difference  of the SDEs for $\xi_{f,a}$, $\xi_{g,b}$ and noting that $(e^{-i \xi_{f,a}}-e^{-i \xi_{g,b}})(dB_1+idB_2)$ has the same distribution as $(e^{-i (\xi_{f,a}-\xi_{g,b})}-1)(dB_1+idB_2)$. The second part of (ii) follows from the first statement. 
%
%
Finally, (\ref{expbnd}) follows from the Markov inequality and the strong Markov property. The existence of the limit is proved in Proposition 9 of \cite{BVBV}. (See that proposition for more details on the proof.)
\end{proof}
%
%

Our main theorems will follow from the following path level large deviations.  \begin{theorem}\label{thm:pathsineb}
Fix $\beta>0$ and let $\alpha_\lambda(t)$ be the process defined in (\ref{sineSDE1}) or (\ref{sineSDE2}). Then the sequence of rescaled processes $(\tfrac{\alpha_\lambda(t)}{\lambda}, t\in [0,\infty))$ satisfies a large deviation principle on $C[0,\infty)$ with scale $\lambda^2$ and good rate function $\Pathsine$. The rate function $\Pathsine$ is defined as 
\begin{equation}
\Pathsine(g) = \int_0^\infty \drift^2(t) \ldp \left( g'(t)/\drift(t)\right)dt, \quad \textup{with} \quad  \drift(t)=\drift_\beta(t)=\tfrac{\beta}{4}e^{-\tfrac{\beta}{4}t} \notag
\end{equation}
in the case where $g(0)=0$, $g'$ exists a.e. and is non-negative. In all other cases  $\Pathsine(g)$ is defined as $\infty$.
\end{theorem}


\begin{theorem}\label{thm:pathschr}
Fix $T>0$ and let $\wt \alpha_\lambda(t)$ be the process defined in (\ref{schrSDE0}). Then the sequence of rescaled processes $(\tfrac{\wt \alpha_\lambda(t)}{\lambda}, t\in [0,T])$ satisfies a large deviation principle on $C[0,T]$ with scale $\lambda^2$ and good rate function $\Pathschr$. The rate function is defined as 
\begin{equation}
\Pathschr(g)=\int_0^T \ldp\left(g'(t)\right) dt\nonumber
\end{equation}
in the case where $g(0)=0$, $g'$ exists a.e. and is non-negative, and  $\Pathschr(g)=\infty$ in all other cases.
\end{theorem}
In order to prove Theorem \ref{thm:pathschr} we observe that $\tfrac{\wt \alpha_\lambda(t)}{\lambda}$ is close to $\tfrac{\lfloor \wt \alpha_\lambda(t)\rfloor_{2\pi}}{\lambda}$ for large $\lambda$ and by Proposition \ref{prop:renewal} we only need to analyze the hitting time $\tau_\lambda$ to understand the evolution of $\lfloor \wt \alpha_\lambda(t)\rfloor_{2\pi}$. The proof of Theorem \ref{thm:pathsineb} will follow along similar lines after approximating the drift in (\ref{sineSDE1}) with a piecewise constant function.

\section{Analysis of the hitting time $\tau_\lambda$}\label{s:hitting}

The following proposition summarizes our bounds on the relevant hitting times.
\begin{proposition}\label{prop:tau}
Let $\tau_{\lambda}=\inf \{t: \wt \alpha_\lambda(t)\ge 2\pi \}$ where $\wt \alpha_\lambda$ is the solution of  (\ref{schrSDE0}) and fix $a<1$. Then  we have 
\begin{align}\label{expmombnd}
E e^{ \frac{\lambda^2 a}{8} \tau_{\lambda}- \frac{\lambda(|a|\wedge  \sqrt{|a|})}{4} \tau_{\lambda}} \leq  e^{-\lambda \II(a)}.
\end{align}
Let $t_a=4 K(a)$ and fix  $0<\eps<|t_a-2\pi|$. Then  we have 
\begin{align}\label{taulower}
P(\lambda \tau_{\lambda} \in [t_a- \varepsilon,t_a+\varepsilon] )&\geq A(\varepsilon,\lambda,a) e^{ -\lambda(\II(a)+\tfrac{a t_a}{8}) -  \lambda  \frac{|a|\eps}{8} -\lambda \frac{|a|}{2}(t_a+\varepsilon)}
\end{align}
where $\lim\limits_{\lambda\to \infty} A(\eps, \lambda, a)=1$  for fixed $a, \eps$. 
\end{proposition}

Our first step is a change of variables in (\ref{schrSDE0}). We introduce $X_\lambda(t)=\log(\tan(\wt \alpha_\lambda(t)/4))$, by It\^o's formula this satisfies the SDE
\begin{align}\label{SDEX}
dX_\lambda= \frac{\lambda}{2} \cosh X_\lambda \ dt+ \frac{1}{2}\tanh X_\lambda \ dt + dB_t, \qquad X_\lambda(0)=-\infty.
\end{align}
The distribution of the hitting time of $2\pi$ for $\wt \alpha_\lambda(t)$ is the same as that of the hitting time of $\infty$ for $X_\lambda$. With a small abuse of notation from now on we will use the notation  $\tau_\lambda$ for the blow-up time of $X_\lambda(t)$, i.e.~$\tau_\lambda =\sup\{t: X_\lambda(t)<\infty \}$. In order to study $\tau_\lambda$ we will introduce a similar diffusion with  a modified drift. Let $a<1$ and consider
\begin{equation}
\label{SDEY}
d Y_{\lambda,a}= \frac{\lambda}{2}\sqrt{ \cosh^2 Y_{\lambda,a} -a}\ dt+ \frac{1}{2} \tanh Y_{\lambda,a} dt+dB_t, \qquad Y_{\lambda,a}(0)=-\infty.
\end{equation}
To prove Proposition \ref{prop:tau} we will choose  an appropriate $a$ and compare $X_\lambda$ with the diffusion $Y_{\lambda,a}$ using the Cameron-Martin-Girsanov formula. Introduce the following notations for the drifts:
\begin{align}
f_\lambda(x)=\frac{\lambda}{2} \cosh x+ \frac{1}{2}\tanh x, \qquad h_{\lambda,a}(y)= \frac{\lambda}{2}\sqrt{ \cosh^2 y- a}\ + \frac{1}{2} \tanh y.\notag
\end{align}
Note, that we have the uniform bound
\begin{align}\label{driftbnd}
\left| f_\lambda(x)-h_{\lambda, a}(x) \right|\le \tfrac12 \lambda |a|. 
\end{align}
The following proposition will be our main tool for our estimates.
\begin{proposition}\label{prop:CMG} 
Fix $a<1$ and consider $X=X_\lambda$ and $Y=Y_{\lambda,a}$. Denote by $\tau_\lambda$ and $\tau_{Y,\lambda}$ the blowup times of $X$ and $Y$. 
 Then for any $s>0$ we have
\begin{align}\label{eq:CMG}
P(\lambda \tau_\lambda>s)=E\left[\ind(\lambda \tau_{Y,\lambda}>s) e^{-G_{s/\lambda}(Y)}\right],
\end{align}
and
\begin{align}\label{expid}
1=E e^{-G_{ \tau \wedge s/\lambda }(Y)}=E e^{G_{ \tau \wedge s/\lambda}(X)},
\end{align}
where
\begin{equation}\notag
G_s(X)=  \int_0^s h_{\lambda,a}(X(t))-f_\lambda(X(t)) dX- \frac{1}{2}\int_0^s (h_{\lambda,a}^2(X)-f_\lambda^2(X))dt.
\end{equation}
\end{proposition}
\begin{proof}
This is just the Cameron-Martin-Girsanov formula for diffusions with explosion. Note, that because of (\ref{driftbnd}) the process $e^{G_{\tau\wedge s}(X)}$ satisfies the Novikov criterion and it is a positive martingale. From this the usual steps of the proof can be completed (see e.g. \cite{KarShr}, \cite{KaratzasRuf2013}).
\end{proof}

\begin{proof}[Proof of Proposition \ref{prop:tau}]
We first estimate the Girsanov exponent  
\begin{align*}
G_s(X)&=   \frac{\lambda}{2} \int_0^s   (\sqrt{ \cosh^2 X-a}- \cosh X) d X\\
&\hspace{2cm} - \frac{1}{2}\int_0^s\left(  -\frac{\lambda^2}{4} a+  \frac{\lambda}{2}( \sqrt{ \cosh^2 X-a}- \cosh X)\tanh X \right)\ dt .
\end{align*}
Applying It\^o's formula for  $\theta(X)=h_{\lambda,a}(X)-f_\lambda(X)$ we have that 
$
\int_0^t \theta(X)dX= \int_{X_0}^{X_t} \theta(x)dx- \frac{1}{2} \int_0^s \theta'(X)dt.
$
 This gives us
\begin{align*}
G_s(X) &=  \frac{\lambda^2 a}{8} s+   \frac{\lambda}{2} \int_{-\infty}^{X_s}  (\sqrt{ \cosh^2 x-a}- \cosh x)dx\\
&\qquad\qquad\qquad+\frac{ \lambda}{4} \int_0^s  \frac{a \tanh}{\sqrt{\cosh^2 X-a}} \cdot \frac{\sqrt{\cosh^2 X-a}-\cosh X}{\sqrt{\cosh^2 X-a}+\cosh X}ds.
\end{align*}
Note, that 
\begin{align*}
\frac12 \int_\rr ( \sqrt{ \cosh^2 x-a}- \cosh x) dx &= - \int_0^{\pi/2} \frac{a}{1+ \sqrt{1-a \sin^2 y}}dy\\
 &= (1-a)K(a)-E(a)=\II(a),
\end{align*}
where this last equality can be seen by differentiating both sides with respect to $a$ and checking equality at $a=0$.  
It is not hard to check that 
\begin{align}\notag
\left|\tfrac{a \tanh}{\sqrt{\cosh^2 X-a}} \tfrac{\sqrt{\cosh^2 X-a}-\cosh X}{\sqrt{\cosh^2 X-a}+\cosh X}\right|\leq \left|\tfrac{a \tanh x}{\sqrt{\cosh^2 x-a}}\right|\le |a|\wedge  \sqrt{|a|}, \qquad \textup{for} \quad a<1
\end{align}
uniformly in $x$. The upper bound $|a|$ follows from $\sqrt{\cosh^2 x-a}\ge |\sinh x|$, while the bound $\sqrt{|a|}$ requires the optimization of the function $\frac{|a| \sqrt{y-1}}{\sqrt{y}\sqrt{y-a}}$ for $y\ge 1$. This gives the bound
\begin{align}\label{Gbnd}
\left| G_{\tau_\lambda}(X)-\frac{\lambda^2 a \tau}{8}-\lambda\II(a)\right|\le \frac{\lambda \tau(|a|\wedge  \sqrt{|a|})}{4}.
\end{align}
To get the exponential moment bound (\ref{expmombnd}) we use $1=E e^{G_{ \tau \wedge s/\lambda}(X)}$ from (\ref{expid}).
We let $s\to \infty$, use  Fatou's lemma  and (\ref{Gbnd}) to get
\begin{align}
1\ge E e^{G_\tau(X)}\ge E e^{ \frac{\lambda^2 a}{8} \tau+\lambda  \II(a)-\frac{\lambda(|a|\wedge \sqrt{|a|} )\tau}{4 }}.
\end{align}
Rearranging the terms we get (\ref{expmombnd}).

To prove the lower bound (\ref{taulower}) we write
\begin{align}\nonumber
&P(\lambda \tau_\lambda\in (t_a-\eps, t_a+\eps))=P(\lambda \tau_{Y,\lambda}>t_a-\eps)- P(\lambda \tau_{Y,\lambda}>t_a+\eps)\\\nonumber 
&\hskip50pt=E\left[\ind(\lambda \tau_{Y,\lambda}>t_a-\eps) e^{-G_{\tau \wedge (t_a-\eps)/\lambda}(Y)}\right]-E\left[\ind(\lambda \tau_{Y,\lambda}>t_a+\eps) e^{-G_{\tau \wedge (t_a+\eps)/\lambda}(Y)}\right]\\\nonumber
&\hskip50pt=E\left[\ind(\lambda \tau_{Y,\lambda}\in (t_a-\eps,t_a+\eps)) e^{-G_{\tau \wedge (t_a+\eps)/\lambda}(Y)}\right]\\
&\hskip50pt=E\left[\ind(\lambda \tau_{Y,\lambda}\in (t_a-\eps,t_a+\eps)) e^{-G_{\tau}(Y)}\right],
\end{align}
where we used the fact that $e^{-G_{\tau \wedge t }(Y)}$ is martingale in the third line. Because of (\ref{Gbnd}) we have 
\begin{align}
G_\tau(Y)\le   \frac{\lambda^2 a \tau}{8} + \lambda \II(a)+\frac{\lambda |a| \tau }{4 },
\end{align}
and we can bound the last expectation as 
\begin{align*}
E\left[\ind(\lambda \tau_{Y,\lambda}\in (t_a-\eps,t_a+\eps)) e^{-G_{\tau}(Y)}\right]&\ge E\left[\ind(\lambda \tau_{Y,\lambda}\in (t_a-\eps,t_a+\eps)) e^{- \frac{\lambda^2 a \tau}{8} - \lambda \II(a)-\frac{\lambda 
{|a|} \tau }{4 }}\right]
\\
&\ge P(\lambda \tau_{Y,\lambda}\in (t_a-\eps,t_a+\eps))  e^{ -\frac{\lambda a (t_a\pm \eps)}{8} - \lambda \II(a)-\frac{\lambda{|a|} (t_a+\eps) }{4}},
\end{align*}
where we choose the sign of $\eps $ in $t_a\pm \eps$ the same way as the sign of $a$. 

If we can show that  $\lim\limits_{\lambda\to \infty}P(\lambda \tau_{Y,\lambda}\in (t_a-\eps,t_a+\eps))   =1$ for fixed $a$ and  $\eps$ then this will complete the proof of (\ref{taulower}).
Note, that $\wt Y(t):=Y_{\lambda, a}(t/\lambda)$ satisfies the SDE
\begin{align*}
d\wt Y=\frac12 \sqrt{\cosh^2 Y-a} dt+\frac{1}{2\lambda}\tanh \wt Y dt+\frac{1}{\sqrt{\lambda}} dB_t, \qquad \wt Y(0)=-\infty.
\end{align*}
As $\lambda\to\infty$, the strong solution of this SDE converges a.s.~to the solution of the ODE 
\[y'=\frac{1}{2}\sqrt{\cosh^2 y- a}, \ \ \ y(0)=-\infty.\]
This ODE is can be  solved and the solutions  satisfies $\int_{-\infty}^{y(t)} \frac{2}{\sqrt{\cosh^2 x-a}} dx=t$. This shows that $y$ explodes exactly at 
\[
\int_{-\infty}^{\infty} \frac{2}{\sqrt{\cosh^2 x-a}} dx=4K(a)=t_a.
\]
This shows that  $\lim\limits_{\lambda\to \infty}P(\lambda \tau_{Y,\lambda}\in (t_a-\eps,t_a+\eps))  = 1$ for fixed $a$ and  $\eps$ and this completes the proof of the proposition.
\end{proof}
We can use the tail estimates of $\tau_\lambda$ to estimate the tail probabilities of $\wt \alpha_\lambda(t)$  for a fixed $t$. Recall the definition of $\ldp(\cdot)$ from (\ref{ldpdef}). 
\begin{lemma}\label{lem:tailbnd}
There exist a constant $c$ so that for $\lambda>2$ we have 
\begin{align}\label{tailbnd}
e^{-\lambda^2 t \ldp(q)+\lambda c(t+1)( \ldp(q)+1)}\ge \begin{cases}
  \, P(\lceil \wt \alpha_\lambda(t)\rceil_{2\pi} \ge q t \lambda) &\qquad\textup{if $q>1$,}\\[4pt]
  \, P(\lfloor \wt \alpha_\lambda(t)\rfloor_{2\pi} \le q t \lambda)& \qquad \textup{if $0<q<1$.}
  \end{cases} 
\end{align}
Moreover, there are absolute constants $c_0,c_1$ so that if $qt\lambda, q$ and $\lambda q \log q$ are all bigger than $c_0$ then
\begin{align}\label{uptail}
 P(\lceil \wt \alpha_\lambda(t)\rceil_{2\pi} \ge q t \lambda)\le e^{-c_1 \lambda^2 t \,q^2 \log q}.
\end{align}
\end{lemma}
\begin{proof}
Introduce the hitting times
\begin{align}\label{def:hit}
\tau_\lambda^{(n)}=\inf\{t>0: \wt \alpha_\lambda(t)>2n\pi\}.
\end{align}
Then by Proposition \ref{prop:renewal} the random variables $\tilde \tau^{(n)}=\tau^{(n)}-\tau^{(n-1)}$ are i.i.d.~with the same distribution as $\tau_\lambda$. Applying the exponential Markov inequality we get
\begin{align}
P(\lfloor \wt \alpha_\lambda(t)\rfloor_{2\pi} \le q t \lambda)&=P\bigg(\sum_{i=1}^{\lceil qt \lambda/(2\pi)\rceil} \tilde \tau^{(i)}\ge  t\bigg)\le \left(E e^{A \tau_\lambda}\right)^{\lceil qt \lambda/(2\pi)\rceil} e^{-At}\label{tail1}
\end{align}
with any $A>0$. Suppose first that $q<1$ which also implies $a=a(q)=K^{-1}(\pi/(2q))\in(0,1)$.   By choosing 
\begin{align}
A=\frac{\lambda^2 a}{8}-\frac{\lambda {|a|}}{4}\label{defA}
\end{align}
we have $A>0$ if $\lambda>2$ and from (\ref{expmombnd}) we have $E e^{A\tau_\lambda}\le e^{-\lambda \II(a)}$.
Together with (\ref{tail1}) this gives
\begin{align}\notag
P(\wt \alpha_\lambda(t)\le q t \lambda)&\le e^{-\lambda \II(a) \lceil qt \lambda/(2\pi)\rceil-(\frac{\lambda^2 a}{8}+\frac{\lambda |a|}{4 })t}\\
&\le e^{-\frac{q t \lambda^2}{2\pi} \II(a) -(\frac{\lambda^2 a}{8}+\frac{\lambda |a|}{4})t+\lambda \left| \II(a)\right|}=e^{-\lambda^2 t \ldp(q)+\lambda\left(\tfrac{|a| t}{4}+|\II(a)|\right)}
\label{xxx1}
\end{align} 
where we used the definitions (\ref{ldpdef}) and (\ref{defH}).

 For the $q>1$ case we use the same steps. Here  $a=K^{-1}(\pi/(2q))<0$ and $A$ defined in (\ref{defA}) is negative which is exactly what we need for the exponential Markov inequality. Eventually we get
\begin{align}\notag
P(\lceil \wt \alpha_\lambda(t)\rceil_{2\pi}\ge q\lambda t) &\le e^{-\lambda \II(a) \lfloor qt \lambda/(2\pi)\rfloor-(\frac{\lambda^2 a}{8}-\frac{\lambda |a|}{4 })t}\\
&\le e^{-\frac{q t \lambda^2}{2\pi} \II(a) -(\frac{\lambda^2 a}{8}-\frac{\lambda |a|}{4 })t+\lambda |\II(a)|}=e^{-\lambda^2 t \ldp(q)+\lambda\left(\tfrac{|a| t}{4}+|\II(a)|\right)}.\label{xxx2}
\end{align}
By Lemma \ref{Hbound} in the Appendix there is a constant $c$ so that 
\begin{align}
\II(a(q))+\tfrac14 |a(q)| t\le c(t+1)( \, \ldp(q)+1),
\end{align}
for all $t, q>0$ which means that we can replace the upper  bounds in (\ref{xxx1}) and (\ref{xxx2}) with $e^{-\lambda^2 t \ldp(q)+\lambda c(t+1)( \ldp(q)+1)}$. This proves the first part of Lemma \ref{lem:tailbnd}. 

For the second part we repeat the same steps as in the $q>1$ case, but now use 
\[
A=\frac{\lambda^2 a}{8}-\frac{\lambda \sqrt{|a|}}{4}.
\]
This gives
\begin{align*}
P(\lceil \wt \alpha_\lambda(t)\rceil_{2\pi}\ge q\lambda t) &\le e^{-\lambda \II(a) \lfloor qt \lambda/(2\pi)\rfloor-\frac{\lambda^2 a}{8}t+\frac{\lambda \sqrt{|a|}}{4 }t}.
\end{align*}
By Proposition \ref{prop:ldpasympt} of the Appendix if $q$ is large enough then $a=K^{-1}(\pi/(2q))>c q^2 \log^2 q$ with some positive constant $c$. If $-a \lambda$ and $q t \lambda$ are big enough (which can be achieved by choosing  $c_0$ big enough),  we will have 
\[
\lfloor qt \lambda/(2\pi)\rfloor>\frac{9}{10}  qt \lambda/(2\pi), \qquad -\frac{\lambda^2 a}{8}+\frac{\lambda \sqrt{|a|}}{4 }<-\frac{11}{10} \cdot \frac{\lambda^2 a}{8}.
\]
Then
\begin{align*}
-\lambda \II(a) \lfloor qt \lambda/(2\pi)\rfloor-\tfrac18 {\lambda^2 a}t+\tfrac14 \lambda \sqrt{|a|} t&< -\lambda^2 t\left(
\tfrac{9}{10}\II(a)\tfrac{q}{2\pi}+\tfrac{11}{10}\tfrac{a}{8}\right)\\
 &=-\lambda^2 t \left(-\frac{7 a}{80}-\frac{9 E(a)}{40 K(a)}+\frac{9}{40}   \right)\\
 &<-c_2 \lambda^2 t q^2 \log^2 q. 
\end{align*}
with a positive constant $c_2$, where in the last step we again used the asymptotics given in  Proposition \ref{prop:ldpasympt} together with (\ref{elliptic1}). This completes the proof of (\ref{uptail}). 
\end{proof}

\section{The path deviation for the $\wt \alpha_\lambda$ process}\label{s:pathschr}

In this section we will prove Theorem \ref{thm:pathschr}. In order to show the large deviation principle we need that
\begin{align*}
\liminf_{\lambda\to \infty} \frac{1}{\lambda^2}P\left(\frac{\wt \alpha_\lambda(\cdot)}{\lambda}\in G\right)&\ge \inf_{g\in G} \Pathschr(g), \qquad \textup{for any open set $G\subset C[0,T]$},\\
\limsup_{\lambda\to \infty} \frac{1}{\lambda^2}P\left(\frac{\wt \alpha_\lambda(\cdot)}{\lambda}\in K\right)&\le \inf_{g\in K} \Pathschr(g), \qquad \textup{for any closed set $K\subset C[0,T]$}.
\end{align*}
The fact that $ \Pathschr(g)$ is a good rate function will be  proved in Proposition \ref{prop:goodrate} of Section \ref{s:goodrate}. 

We will use the fact that $\ldp(x)$ is strictly convex on $(0,\infty)$ with a global minimum at $\ldp(1)=0$, and also that there is a constant $c>0$ so that 
\begin{align}\label{ldptail1}
c^{-1} \le \frac{\ldp(x)}{x^2 \log^2 x}\le c, \qquad \textup{for all $x>2$}.  
\end{align}
These statements will be proved in Propositions \ref{prop:ldpconvex} and \ref{prop:ldpasympt} of the Appendix.

\begin{proof}[Proof of the large deviations upper bound in Theorem \ref{thm:pathschr}]
We will follow the standard strategy for proving path level large deviations. Consider a closed subset $K$ of $C[0,T]$.  We need to bound $P(\tfrac{1}{\lambda}\wt \alpha_\lambda(\cdot)\in K)$.  Define the  $\delta$-`fattening' of $K$ as
\begin{align}\label{Kfat}
K^\delta:=\{f\in C[0,T]: \|f-g\|\le \delta \textup{ for some }g\in K\}. 
\end{align}
From now on $\| \cdot \|$ denotes the sup-norm on the appropriate interval. 

Let $\pi_N$ be the following projection of $C[0,T]$ to piecewise linear paths: 
\begin{align}\label{pi}
(\pi_N f)(i T/N)=\lfloor f(iT /N)\rfloor_{2\pi}, \qquad 0\le i\le N
\end{align}
and $\pi_N f$ is defined  linearly between these points. Then
\begin{align}\label{pathbnd1}
P( \tilde \alpha_\lambda/ \lambda \in K) & \leq P( \|\tilde \alpha_\lambda- \pi_N \tilde \alpha_\lambda \| \geq \delta \lambda)+ P\big( \pi_N ( \tfrac1{\lambda}\tilde \alpha_\lambda) \in K^\delta \big).
\end{align}
We will bound the two probabilities in (\ref{pathbnd1}) separately. 

The first term can be rewritten as
\begin{align}\label{12345}
P\left[ \left\|{\tilde \alpha_\lambda}-{ \pi_N \tilde \alpha_\lambda}\right \| \geq \delta\lambda \right]&= P\bigg( \max_k  \sup\limits_{t\in [\frac{(k-1)T}{N},\frac{kT}{N}]}\left|{\pi_N \tilde \alpha_\lambda(t)}-{ \tilde \alpha_\lambda(t)}\right|\geq \delta\lambda \bigg).
\end{align}
By Proposition \ref{prop:renewal} the process $\lfloor \wt \alpha_\lambda(t)\rfloor_{2\pi}$ is non-decreasing.Thus  for any fixed $k$ we have 
\begin{align*}
\sup\limits_{t\in [\frac{(k-1)T}{N},\frac{kT}{N}]}\left|{\pi_N \tilde \alpha_\lambda(t)}-{ \tilde \alpha_\lambda(t)}\right|  \le {\lceil \wt \alpha_\lambda ({(k+1)T}/{N})\rceil_{2\pi}}- {\lfloor \wt \alpha_\lambda({kT}/{N})\rfloor_{2\pi}}.
\end{align*}
By Proposition \ref{prop:stochdom} the term on the right is stochastically dominated by
$
\wt\alpha_\lambda(T/N)+4\pi
$
therefore
\begin{align}
 P( \|\tilde \alpha_\lambda- \pi_N \tilde \alpha_\lambda \| \geq \delta\lambda )&\le N P(\wt \alpha_\lambda(T/N)+4\pi\ge \delta \lambda)\le N P\left(\frac{1}{\lambda}\wt \alpha_\lambda(T/N)\ge \frac{\delta }{2} \right)\label{23456}
\end{align}
where the last bound holds if $\lambda>8\pi/\delta$. Using Lemma \ref{lem:tailbnd} we get
\begin{align}\notag
N P\left(\frac{1}{\lambda}\wt \alpha_\lambda(T/N)\ge \frac{\delta }{2} \right)&\le N e^{-(\lambda^2 \frac{T}{N}+ \lambda c_1(T/N+1)) \ldp\left(\frac{\delta N}{2T}\right)+\lambda c_1(T/N-1) }
\end{align}
and this leads to
\begin{align}
\limsup_{\lambda\to \infty} \frac{1}{\lambda^2} \log  P( \|\tilde \alpha_\lambda- \pi_N \tilde \alpha_\lambda \| \geq \delta \lambda)\le -\frac{T}{N}\ldp\left(\frac{\delta N}{2T}\right)\label{1st_term}.
\end{align}
Note, that for fixed $\delta$ and $T$ as $N\to \infty$ the right hand side converges to $-\infty$ by (\ref{ldptail1}). 

The second term on the right side of (\ref{pathbnd1}) can be bounded as 
\begin{align}\notag
P\left(\pi_N (\wt \alpha_\lambda/\lambda)\in K^\delta\right)\le P\left(\Pathschr(\pi_N(\wt \alpha_\lambda/\lambda ))\ge \inf_{g\in K^\delta} \Pathschr(g)   \right).
\end{align}
We introduce 
\begin{align}\notag
\Delta \wt \alpha_i=\frac{N}{\lambda T}\left(\lfloor \wt \alpha(i T/N)\rfloor_{2\pi}-\lfloor \wt \alpha((i-1) T/N)\rfloor_{2\pi}\right), \qquad  \textup{for $1\le i \le N$} 
\end{align}
and $C_\delta=\inf_{g\in K^\delta} \Pathschr(g)  $. Then we have to bound
\begin{align}
P(\Pathschr(\pi_N(\wt \alpha_\lambda/\lambda ))\ge C_\delta)=P\left(\sum_{i=1}^N \frac{T}{N}\ldp\left({\Delta \wt \alpha_i}\right)\ge C_\delta\right).
\end{align}
We can apply Proposition \ref{prop:stochdom} with $t_i=\frac{iT}{N}, 1\le i\le N$ to get independent random variables $\xi_i$ with $\xi_i\eqd \wt \alpha_\lambda(T/N)$ and
\begin{align}\notag
\frac{N}{\lambda T}\lfloor \xi_i \rfloor_{2\pi} \le \Delta \wt \alpha_i\, \le \,\frac{N}{\lambda T}\left(\lfloor \xi_i \rfloor_{2\pi}+2\pi\right).
\end{align}
%
%
%
%
Because of the convexity of $\ldp(\cdot)$ we then  have
\begin{align*}
\ldp\left({\Delta \wt \alpha_i}\right)&\le \max\left( \ldp\left(\frac{N\lfloor\xi_i\rfloor_{2\pi}}{ \lambda T}\right), \ldp\left(\frac{N  \lfloor \xi_i \rfloor_{2\pi}}{\lambda T}+\frac{2\pi N}{\lambda T}\right) \right)\\
&\le (1+\frac{2\pi N}{\lambda T})\ldp\left(\frac{N \lfloor \xi_i \rfloor_{2\pi}}{ \lambda T}\right)+c \frac{2\pi N}{\lambda T}
\end{align*}
where we used  Lemma \ref{Iexpansion} of the Appendix for the last bound. Fix $1/2>\eps>0$. Using the exponential Markov inequality, the independence of $\xi_i$ and $\xi_i\eqd \wt \alpha_\lambda(T/N)$ we get the bound
\begin{align}
P\left(\sum_{i=1}^N \frac{T}{N}\ldp\left({\Delta \wt \alpha_i}\right)\ge C_\delta\right)&\le \left(E e^{(1-2\eps) \lambda^2 \frac{T}{N}\left((1+ \frac{2\pi N}{\lambda T})\ldp\left(\frac{N \lfloor\wt \alpha_\lambda(T/N)\rfloor_{2\pi}}{ \lambda T}\right)+c \frac{2\pi N}{\lambda T}\right)}\right)^N e^{-(1-2\eps) \lambda^2 C_\delta}\notag\\
&\le \left(E e^{(1-\eps) \lambda^2 \frac{T}{N} \ldp\left(\frac{N \lfloor\wt \alpha_\lambda(T/N)\rfloor_{2\pi}}{ \lambda T}\right)}\right)^N e^{(1-2\eps)c 2\pi\lambda  N-(1-2\eps)\lambda^2C_\delta},\label{upper1}
\end{align}
where the second inequality holds for fixed $\eps, N, T$ if $\lambda$ is big enough. 
Our next step is to estimate the exponential moment $E e^{(1-\eps) \lambda^2  \frac{T}{N}\ldp\left(\frac{N\lfloor\wt \alpha_\lambda(T/N)\rfloor_{2\pi}}{ \lambda T}\right)}$ for a fixed $\eps>0$. By Lemma \ref{lem:expmom} below  if $N,T, \eps$ are fixed then
\[
\limsup_{\lambda\to \infty} \frac1{\lambda^2} \log E e^{(1-\eps) \lambda^2  \frac{T}{N}\ldp\left(\frac{N\lfloor\wt \alpha_\lambda(T/N)\rfloor_{2\pi}}{ \lambda T}\right)}\le  0. 
\]
Using this with (\ref{upper1}) we get  
\begin{align}\label{2nd}
\limsup_{\lambda\to \infty} \frac{1}{\lambda^2} \log P\left(\sum_{i=1}^N \frac{T}{N}\ldp\left({\Delta \wt \alpha_i}\right)\ge C_\delta\right)\le -(1-2\eps)C_\delta.
\end{align}
Now we let  $\varepsilon\to 0$ and then $N\to \infty$.
The bounds  (\ref{1st_term}), (\ref{2nd}) with (\ref{pathbnd1}) give
\begin{align}
\limsup_{\lambda\to \infty} \frac1{\lambda^2} \log P(\lambda^{-1}\wt\alpha_\lambda(\cdot) \in K)\le -\inf_{g\in K^\delta} \Pathschr(g). 
\end{align}
Using the fact that $\Pathschr$ is a good rate function (which is proved in Proposition \ref{prop:goodrate} of Section \ref{s:goodrate}) we get that the right hand side converges to $-\inf_{g\in K} \Pathschr(g)$ as $\delta\to 0$. (See e.g.~Lemma 4.1.6 from \cite{DemboZeitouni}.) This  finishes the proof of the lower bound.
\end{proof}
Now we will prove the missing estimate for the lower bound. 
\begin{lemma}\label{lem:expmom}
Fix  $t>0$ and $1>\eps>0$.  Then
\[
\limsup_{\lambda\to \infty} \frac{1}{\lambda^2} \log E e^{(1-\eps) \lambda^2  t \ldp\left(\frac{\lfloor\wt \alpha_\lambda(t)\rfloor_{2\pi}}{ \lambda t}\right)}\le 0.
\]
\end{lemma}
\begin{proof}
 Introduce the temporary notation $G(x)= \lambda^2  t \ldp\left(x\right)$. This is a convex function with $G(1)=0$ as its minimum. 
Then we have 
\begin{align*}
E e^{(1-\eps) \lambda^2 t  \ldp\left(\frac{ \lfloor\wt \alpha_\lambda(t)}{ \lambda t}\right)\rfloor_{2\pi}}\le&
\ 2-\int_0^{1} (1-\eps)G'(x) e^{(1-\eps)G(x)} P({ \lfloor\wt \alpha_\lambda(t)\rfloor_{2\pi}}<\lambda t x)dx\\
&\hspace{1.5cm} +\int_{1}^\infty (1-\eps)G'(x) e^{(1-\eps)G(x)} P({ \lfloor\wt \alpha_\lambda(t)\rfloor_{2\pi}}>{ \lambda t}x)dx.
\end{align*}
Using Lemma \ref{lem:tailbnd}  we get
\[
 P({ \lfloor\wt \alpha_\lambda(t)\rfloor_{2\pi}}<{ \lambda t}x) \leq \exp \big\{ - (1-c_1\lambda^{-1}(1+ t^{-1}))G(x)+ \lambda c_1 (t+1) \big\}
\]
for $x<1$ and a similar bound for  $ P({ \lfloor\wt \alpha_\lambda(t)\rfloor_{2\pi}}>{ \lambda t} x) \le  P({ \lceil\wt \alpha_\lambda(t)\rceil_{2\pi}}>{ \lambda t}x) $ for $x>1$. This gives us
\begin{align*}
E e^{(1-\eps) \lambda^2t \ldp\left(\frac{\wt \alpha_\lambda(t)}{ \lambda t}\right)}\le&
\ 2-\int_0^{1} (1-\eps)G'(x) e^{((1+t^{-1})(c_1/\lambda)-\eps)G(x)+\lambda c_1(t+1)} dx\\
&\hspace{1cm}+\int_{1}^\infty (1-\eps)G'(x) e^{((1+t^{-1})(c_1/\lambda)-\eps)G(x)+\lambda c_1(t+1)} dx\\
\le& 2+ 4\eps^{-1} e^{\lambda\, c_1(t+1)} 
\end{align*}
where the last inequality holds if  $(1+t^{-1})c_1/\lambda<\eps/2$, i.e.~for large enough $\lambda$.  From this the lemma  follows. 
\end{proof}

Now we turn to the lower bound proof in the large deviation result of Theorem \ref{thm:pathschr}. As we will see, we will be able to reduce the problem to studying the probability of $\frac1{\lambda} \wt \alpha_\lambda(t)$ being close to a straight line.
\begin{proposition}\label{prop:lowerlin}
Fix $q> 0$ and $T, \eps>0$. Then
\begin{align}
\lim\limits_{\eps\to 0}\liminf_{\lambda\to \infty} \frac{1}{\lambda^2} \log P(\wt \alpha(t) \in [\lambda(q t-\eps), \lambda(q t+\eps)], t\in [0,T])&\ge -T \ldp(q).\notag
\end{align}
\end{proposition}  
\begin{proof}
%
 As $q>0$ we may assume $\eps \le q T/2$ by choosing $\eps$ small enough. Let $N=\frac{\lceil (qT+\epsilon)\lambda\rceil_{2\pi}}{2\pi}$, and choose $\eps_1= \frac{\pi \eps}{2q(qT+\eps)}$, which satisfies $\eps_1< \frac{\eps \lambda}{2qN}$ for $\lambda>2$.
Recall the definition of $\tau^{(n)}_\lambda$ and $\wt \tau^{(n)}_\lambda$ from (\ref{def:hit}). We will prove that 
\begin{align}\notag
P(\wt \alpha(t) \in [\lambda(q t-\eps), &\lambda(q t+\eps)], t\in [0,T])\\
&\ge P\left( \lambda \wt \tau^{(k)}_\lambda\in ( \tfrac{2\pi}{q}-\eps_1,  \tfrac{2\pi}{q}+\eps_1), 1\le k \le N\right).\notag
\end{align}
Roughly speaking, this will follow from the simple fact that if we are within $\eps/q$ of the line $y=qt$ in the horizontal  direction, then we are within $\eps$ in the vertical direction. 
If $ \lambda \wt \tau^{(k)}_\lambda\ge \frac{2\pi}{q}-\eps_1$ for $1\le k \le N$ then $\lambda \tau^{(k)}_\lambda\ge k( \frac{2\pi}{q}-\eps_1)$ and 
\begin{align*}
\wt \alpha_\lambda\left( \tfrac{k}{\lambda}\left( \tfrac{2\pi}{q}-\eps_1\right)\right)\le 2 k \pi=
\lambda \frac{2\pi}{2\pi/q-\eps_1}\cdot 
\frac{k}{\lambda}\left( \tfrac{2\pi}{q}-\eps_1\right) 
\end{align*}
for $1\le k \le N$. Together with the fact that $\lfloor \wt \alpha_\lambda \rfloor_{2\pi}$ is non-decreasing we get that
\begin{align*}
\wt \alpha_\lambda(t)\le \lambda \frac{2\pi}{2\pi/q-\eps_1}\cdot \left(t+\tfrac{1}{\lambda}(2\pi/q-\eps_1)   \right), \qquad \textup{ for } t\le \tfrac{N}{\lambda}\left( \tfrac{2\pi}{q}-\eps_1\right).
\end{align*}
This inequality implies $\wt \alpha_\lambda(t)\le \lambda (qt+\eps)$, for $ t\le T$, $\lambda \eps >4\pi$. 

The other direction is similar, if we have $ \lambda \wt \tau^{(k)}_\lambda\le \frac{2\pi}{q}+\eps_1$ for $1\le k \le N$ then
\begin{align*}
\wt \alpha_\lambda\left( \tfrac{k}{\lambda}\left( \tfrac{2\pi}{q}+\eps_1\right)\right)\ge 2 k \pi=
\lambda \frac{2\pi}{2\pi/q+\eps_1}\cdot 
\frac{k}{\lambda}\left( \tfrac{2\pi}{q}+\eps_1\right) 
\end{align*}
which implies 
\begin{align*}
\wt \alpha_\lambda(t)\ge \lambda \frac{2\pi}{2\pi/q+\eps_1}\cdot \left(t-\tfrac{1}{\lambda}(2\pi/q+\eps_1)   \right), \qquad \textup{ for } t\le \tfrac{N}{\lambda}\left( \tfrac{2\pi}{q}+\eps_1\right).
\end{align*}
and $\wt \alpha_\lambda(t)\ge \lambda (qt-\eps)$ for $t\le T$. Using the independence of $\wt \tau^{(k)}_\lambda$ we get the bound
\begin{align}
P(\wt \alpha(t) \in [\lambda(q t-\eps), \lambda(q t+\eps)], t\in [0,T])\ge P\left( \lambda \wt \tau^{(k)}_\lambda\in ( \tfrac{2\pi}{q}-\eps_1,  \tfrac{2\pi}{q}+\eps_1)\right)^{N}.
\end{align}
By the lower bound (\ref{taulower}) we have
\begin{align*}
\log P(\wt \alpha(t) \in [\lambda(q t-\eps), \lambda(q t+\eps)], t\in [0,T])&\\ 
&\hskip-150pt \ge \frac{\lambda(qT+2\eps)}{2\pi}\left(- \frac{2\pi \lambda}{q} \ldp(q)- \frac{\lambda |a|}{8}(\eps_1+4(2\pi/q+\eps_1))+ \log A(\eps_1, \lambda, a)\right).
\end{align*}
Recalling $\eps_1= \frac{\pi \eps}{2q(qT+\eps)}$ we get
\begin{align*}
\lim\limits_{\eps\to 0}\liminf_{\lambda\to \infty} \frac{1}{\lambda^2} \log P(\wt \alpha(t) \in [\lambda(q t-\eps), \lambda(q t+\eps)], t\in [0,T])&\ge -T \ldp(q).\qedhere
\end{align*}
\end{proof}

\begin{proof}[Proof of the lower bound in Theorem \ref{thm:pathschr}.]
Let $G$ be an open subset of $C[0,T]$. We would like to show that 
\begin{align}\label{LDPlowerbnd}
\liminf_{\lambda\to \infty} \frac{1}{\lambda^2} \log P(\frac{1}{\lambda} \wt \alpha_\lambda(\cdot) \in G)\ge - \inf_{g\in G} \int_0^T \ldp\left(g'(t)\right) dt.
\end{align}
For this it is enough to prove that for any $g\in G$ with $\int_0^T \ldp\left(g'(t)\right) dt<\infty$ and  $\delta>0$ we have 
\begin{align}
\liminf_{\lambda\to \infty} \frac{1}{\lambda^2} \log P(\frac{1}{\lambda} \wt \alpha_\lambda(\cdot) \in G)\ge -\int_0^T \ldp\left(g'(t)\right) dt-\delta.\label{openlwr}
\end{align}
We can approximate $g$ with a piecewise linear function $\tilde g$ in the sup-norm so that we have $|\int_0^T \ldp\left(g_n'(t)\right) dt- \int_0^T \ldp\left(g'(t)\right) dt|<\delta$. Because of this  we may assume that $g$ is piecewise linear, moreover, we may assume that there are no horizontal segments in $g$. Suppose that $g$ is linear with slope $q_i$ on the interval $[T_i,T_{i+1}]$ with $0\le i \le k-1$ and $0=T_0<T_1<\dots<T_k=T$. We claim that if $\lambda>\lambda_0(\eps, k)$ then
\begin{align}\nonumber
P(\|\frac{1}{\lambda} \wt \alpha_\lambda(\cdot) -g(\cdot)\|\le \eps)&\ge
P\left(|\frac{1}{\lambda}( \wt \alpha_\lambda(t)- \wt \alpha_\lambda(T_i)) -q_i (t-T_i)|\le \eps/k,\textup{ if } t\in[T_i,T_{i+1}]\right)
\\
&\ge\label{second}
 \prod_{i=0}^{k-1} P\left(|\frac{1}{\lambda} \wt \alpha_\lambda(t) -q_i t|\le \eps/(2k),\textup{ for } t\in[0,T_{i+1}-T_i]\right)
\end{align}

The first inequality is straightforward, to prove the second we use the coupling in the proof of  Proposition \ref{prop:stochdom}. Recall the definition of the processes $\hat \alpha_i(s)$ defined on $[t_{i-1},t_i]$. These were independent for different values of $i$ and the process $\hat \alpha_i(s+t_{i-1}), s\in[0,t_i-t_{i-1}]$ had the same distribution as $\wt \alpha_\lambda(s), s\in[0,t_i-t_{i-1}]$. We also had
\begin{align*}
\hat \alpha_i(s)+\lfloor \wt \alpha_\lambda(t_{i-1})\rfloor_{2\pi} \le\wt \alpha_\lambda(s)\le  \hat \alpha_i(s)+\lfloor \wt \alpha_\lambda(t_{i-1})\rfloor_{2\pi}+2\pi.
\end{align*}
for $s\in [t_{i-1},t_i]$. By choosing $\lambda>\lambda_0=4\pi k/\eps$ the inequality (\ref{second}) follows by the independent increment property of the Brownian motion. 

By Proposition \ref{prop:lowerlin} we have the bound
\begin{align*}
\lim\limits_{\eps\to 0} \liminf_{\lambda\to \infty} \tfrac{1}{\lambda^2} \log P(\|\frac{1}{\lambda} \wt \alpha_\lambda(\cdot) -g(\cdot)\|\le \eps)\ge 
-\sum_{i=0}^{k-1} (T_{i+1}-T_i) \ldp(q_i)= -\int_0^T \ldp\left(g'(t)\right) dt
\end{align*}
from which (\ref{openlwr})  and thus the proof of the lower bound follows. 
\end{proof}



\section{The path deviation for the $\Sineb$ process}\label{s:pathsine}

This section contains the proof of Theorem \ref{thm:pathsineb}.
The strategy for the proof  is to approximate the SDE (\ref{sineSDE2}) with a version where the drift is piecewise constant and then use elements of the proof of Theorem \ref{thm:pathschr}. Just as in the proof of Theorem {\ref{thm:pathschr}, we need to show an upper and a lower bound to prove the large deviation principle. The fact that $\Pathsine$ is a good rate function will be proved in Proposition \ref{prop:goodrate} of Section \ref{s:goodrate}.

\begin{proof}[Proof of the upper bound in Theorem \ref{thm:pathsineb}]

For the proof of the upper bound we go through a series of approximations: we essentially cut of the tail of the process, then replace the drift in the SDE with a piecewise constant version and then approximate the process with a piecewise linear version. Recall that $\alpha_\lambda(t)$ solves the SDE (\ref{sineSDE1}) and that we introduced the notation $\drift(t)=\tfrac{\beta}{4}e^{\tfrac{\beta}{4}t}$. Fix $T>0$, the value of which will go to infinity later. The first approximating process is defined as 
\begin{align*}
\alpha_\lambda^{(1)} (t)& = \alpha_\lambda (t) \ind(t\leq T)+ (\alpha_\lambda(T)+ \lambda(e^{-\frac{\beta}{4}T} - e^{-\frac{\beta}{4}t})) \ind(t>T),
\end{align*}
this solves the SDE (\ref{sineSDE1}) with the noise `turned off' at $t=T$. For the second process we define
\begin{align}\label{Pproj}
\drift_N(t)= \drift(Ti/N), \qquad  t\in [Ti/N,T(i+1)/N)
\end{align}
and consider the solution $\xi_{\lambda \drift_N}$ of (\ref{xi}) with drift $\lambda \drift_N$ and initial condition 0. Let 
\begin{align}
\alpha_\lambda^{(2)}(t) & = \xi_{\lambda \drift_N} (t) \ind(t\leq T)+ ( \xi_{\lambda \drift_N} (T)+ \lambda(e^{-\frac{\beta}{4}T} - e^{-\frac{\beta}{4}t})) \ind(t>T).\notag
\end{align}
Finally, let $\pi_{MN}$ is the projection defined in (\ref{pi}) with intervals of size $T/MN$, that is $\pi_{MN} f$ is the piecewise linear path that satisfies
\[ 
(\pi_{MN} f)(Ti/(MN)) = \lfloor f(Ti/(MN))\rfloor_{2\pi},
\]
and is linear between these values. Define
\begin{align}
\alpha_\lambda^{(3)}(t)& = \pi_{MN} \xi_{\lambda \drift_N} (t) \ind(t\leq T)+ (\pi_{MN} \xi_{\lambda \drift_N} (T)+ \lambda (e^{-\frac{\beta}{4}T} - e^{-\frac{\beta}{4}t})) \ind(t>T).\notag
\end{align}
Then for any closed set $K\subset C[0,\infty)$ we have that 
\begin{align}
P\left( \frac{\alpha_\lambda}{\lambda} \in K\right) & \leq P\Big( \frac{\alpha_\lambda^{(3)}}{\lambda} \in K^{3\delta} \Big)+ P(\|\alpha_\lambda^{(1)}- \alpha_\lambda \|_\infty \geq \delta \lambda )\notag \\
& \hspace{2.5cm}+ P(\|\alpha_\lambda^{(2)}- \alpha_\lambda^{(1)} \|_\infty \geq \delta \lambda )+P(\|\alpha_\lambda^{(3)}- \alpha_\lambda^{(2)} \|_\infty \geq \delta \lambda ),\label{together}
\end{align}
where $K^{3\delta}$ is defined similarly to (\ref{Kfat}), as the $3\delta$-fattening of $K$. 
We will begin with the main term.  Let 
\begin{equation}\notag
\mathcal{J}_N (g)= \int_0^\infty \drift_N^2(t) \ldp \left( \frac{g'(t)}{\drift_N(t)}\right)dt,
\end{equation}
and define (similarly to the $\tilde \alpha_\lambda$ case in the proof of Theorem \ref{thm:pathschr})
\[
\Delta \alpha_i = \frac{MN}{\lambda \drift_N(\tfrac{T i}{MN})T} \left( \lfloor \alpha^{(3)}(Ti/(MN))\rfloor_{2\pi}- \lfloor \alpha^{(3)}(T(i-1)/(MN))\rfloor_{2\pi} \right), \text{ for } 1\leq i \leq MN .
\]
Then,
\begin{align*}
P\Big( \frac{\alpha_\lambda^{(3)}}{\lambda} \in K^{3\delta} \Big)& \leq P\Bigg( \mathcal{J}_N\bigg( \frac{\alpha_\lambda^{(3)}}{\lambda} \bigg) \geq \inf_{g\in K^{3\delta}} \mathcal{J}_N(g) \Bigg)\\
&= P \Bigg( \sum_{i=1}^{ MN} \frac{T(\drift_N(Ti/(MN)))^2}{MN} \ldp (\Delta \alpha_i) \geq \inf_{g\in K^{3\delta}} \mathcal{J}_N(g) \Bigg).
\end{align*}
Take $\widehat \alpha_i$ to solve (\ref{schrSDE0}) but with the Brownian motion $B(t+ Ti/(MN))-B(T i/(MN))$ and $\lambda_i= \lambda \drift_N(Ti/(MN))$.  Then using the same arguments as in the  bound  (\ref{upper1}) we get
\begin{align}
P\Big( \frac{\alpha_\lambda^{(3)}}{\lambda} \in K^{3\delta} \Big)& \leq e^{-(1-\eps) \lambda^2 C_{\delta,N}}
\prod_{i=1}^{ MN} \left(E e^{(1-\eps) \lambda_i^2 \frac{T}{MN}\left((1+ \frac{2\pi MN}{\lambda T})\ldp\left(\frac{ \lfloor \widehat \alpha_i(T/(MN))\rfloor_{2\pi}}{ \lambda_i (T/MN) }\right)+c_2 \frac{2\pi M N}{\lambda_i T}\right)}\right) \notag
\end{align}
where $C_{\delta,N}=\inf_{g\in K^{3\delta}} \mathcal{J}_N(g)$.  Using the bound proved in Lemma \ref{lem:expmom} we  get that 
\begin{align}
\limsup_{\lambda\to \infty} \frac{1}{\lambda^2}\log P\Big( \frac{\alpha_\lambda^{(3)}}{\lambda} \in K^{3\delta} \Big)& \leq 
-(1-\eps)C_{\delta,N}\label{main}
\end{align}
We now turn to the first error term.  Using the fact that $\lfloor \alpha_\lambda \rfloor_{2\pi}$ is non-decreasing (which follows from (i) of  Proposition \ref{prop:SDEprop}) we get that 
\[
\|\alpha_\lambda^{(1)}-\alpha_\lambda\|\le \alpha_\lambda(\infty)-\alpha_\lambda(T)+\lambda e^{-\tfrac{\beta}{4}T},
\]
where $\alpha_\lambda(\infty)$ is the limit of $\alpha_\lambda(t)$ as $t\to \infty$.   Choose $T$ large enough so that $e^{-\frac{\beta}{4}T} \leq \delta/2$.  Then
\begin{align}
P(\| \alpha_\lambda^{(1)}- \alpha_\lambda\| \geq \delta \lambda)& \leq P(\alpha_\lambda(\infty)- \alpha_\lambda (T) \geq \delta \lambda/2)\notag
\end{align}
We will deal with this tail probability in Proposition \ref{tailbound} below. In particular, we will show that there is a constant $c_1>0$ so that
\begin{align}
\limsup_{\lambda\to \infty}\frac{1}{\lambda^2} \log P(\| \alpha_\lambda^{(1)}- \alpha_\lambda\| \geq \delta \lambda)\leq -c_1   T \delta^2.\label{error1}
\end{align}
For the second error term we first note that $ \| \alpha_\lambda^{(2)}- \alpha_\lambda^{(1)}\| =\sup_{t\in[0,T]} | \alpha_\lambda^{(2)}(t)- \alpha_\lambda^{(1)}(t)|$.  Using the coupling of Proposition \ref{prop:SDEprop} we can show that on $[0,T]$ the process $\alpha_\lambda^{(2)}- \alpha_\lambda^{(1)}$ will have the same distribution as the solution of the SDE (\ref{xi}) with initial condition 0 and drift $\lambda(\drift_N-\drift)\ge 0$. Moreover, this process will be non-negative (because the drift is non-negative), and since $\lambda(\drift_N(t)-\drift(t))\le \lambda\tfrac{\beta T}{4N}$ for $t\in[0,T]$, it will be bounded 
by the solution of the SDE (\ref{xi}) with a constant drift $ \lambda\tfrac{\beta T}{4N}$.
 Because of this $ \| \alpha_\lambda^{(2)}- \alpha_\lambda^{(1)}\|$ is stochastically bounded by $\sup_{t\in[0,T]} \wt \alpha_{\lambda\tfrac{\beta T}{4N}}(t)\le\wt \alpha_{\lambda\tfrac{\beta T}{4N}}(T)+2\pi$ with $\wt \alpha_\lambda$ from (\ref{schrSDE0}, using the fact that $\lfloor \wt \alpha_\lambda(t)\rfloor_{2\pi}$ is non-decreasing. Thus for $\delta \lambda>4\pi$ we have
\begin{align}
P( \| \alpha_\lambda^{(2)}- \alpha_\lambda^{(1)}\| \geq \delta \lambda)&  \leq P( \tilde \alpha_{\lambda\tfrac{\beta T}{4N}}(T) \geq \tfrac12 \delta \lambda).\notag
\end{align}
 If $N$ and $T$ are fixed then if $\lambda$ is big enough then we can  apply Lemma \ref{lem:tailbnd} for the right hand side with $\tilde \lambda=\lambda\tfrac{\beta T}{4N}$, $t=T$ and $q=\frac{\tfrac12 \delta \lambda}{T \lambda\tfrac{\beta T}{4N}}=\frac{2\delta N}{\beta T^2}$. This leads to
\begin{equation}
\limsup_{\lambda\to \infty} \frac{1}{\lambda^2} \log P( \| \alpha_\lambda^{(2)}- \alpha_\lambda^{(1)}\| \geq \tfrac12 \delta \lambda) \leq
-\frac{\beta^2 T^3}{4^2N^2}\ldp \left( \frac{2\delta  N}{\beta T^2} \right).\label{error2}
\end{equation}
For the third error term we first note that 
\[
\| \alpha_\lambda^{(3)}-\alpha_\lambda^{(2)}\|\le \sup_{t\in[0,T]} |\alpha_\lambda^{(3)}(t)-\alpha_\lambda^{(2)}(t)|\le
\max_{i} \sup_{t\in [Ti/N,T(i+1)/N]} |\alpha_\lambda^{(3)}(t)-\alpha_\lambda^{(2)}(t)|,
\]
and thus
\[
P(\| \alpha_\lambda^{(3)}-\alpha_\lambda^{(2)}\| \geq \delta \lambda) \le \sum_{i=0}^{N-1} P\left( \sup_{t\in [Ti/N,T(i+1)/N]} |\alpha_\lambda^{(3)}(t)-\alpha_\lambda^{(2)}(t)|\ge \delta \lambda\right).
\]
In the interval $[Ti/N,T(i+1)/N]$ the process $\alpha_\lambda^{(2)}$ solves the SDE (\ref{schrSDE0}) with constant drift $\lambda \drift_N(Ti/N)$. Here we can use the same steps that we used in the proof of Theorem \ref{thm:pathschr} between (\ref{12345}) and (\ref{23456}) to get 
\begin{align}
P(\| \alpha_\lambda^{(3)}-\alpha_\lambda^{(2)}\| \geq \delta \lambda) &\leq  \sum_{i=1}^{ N-1} M P\left( \tilde \alpha_{\lambda \drift_N(T i/N)}(T/(MN)) \geq \delta \lambda/2\right)\notag\\\notag
&\le MN P\left( \tilde \alpha_{\tfrac{\beta}{4}\lambda }(T/(MN)) \geq \delta \lambda/2\right)
\end{align}
for $\lambda$ big enough compared to $\delta^{-1}$. For large enough $\lambda$ we can apply Lemma \ref{lem:tailbnd} for the right hand side with $\tilde \lambda=\tfrac{\beta}{4}\lambda$, $t=T/(MN)$ and $q=\tfrac{2\delta M N}{\beta T}$ to get
\begin{align}
\limsup_{\lambda\to \infty} \frac{1}{\lambda^2}\log P(\| \alpha_\lambda^{(3)}-\alpha_\lambda^{(2)}\| \geq \delta \lambda) \le
-\frac{\beta^2}{4^2}\frac{T}{MN} \ldp\left(\frac{2\delta MN}{\beta T}\right). \label{error3}
\end{align}
Now taking (\ref{together}) with the bounds (\ref{main}), (\ref{error1}), (\ref{error2}) and (\ref{error3}) we get
\begin{align}
\label{preveq}
\limsup_{\lambda\to \infty} &\frac{1}{\lambda^2} \log P \Big( \frac{\alpha_\lambda}{\lambda} \in K \Big) \\
& \leq \max \left\{ -(1-\varepsilon)C_{\delta,N}, -c_1 T \delta^2 ,-\frac{\beta^2 T^3}{4^2N^2}\ldp \left( \frac{2\delta  N}{\beta T^2} \right) ,-\frac{\beta^2}{4^2}\frac{T}{MN} \ldp\left(\frac{2\delta MN}{\beta T}\right)\right\}. \notag
\end{align}
Taking $N$ to $\infty$ the last two terms go to $-\infty$ (using the bounds (\ref{ldptail1})) while the  first term converges to $(1-\eps)C_\delta^T$ with 
\[
C_{\delta}^T= \inf_{g\in K^{3\delta}}\int_0^T \drift^2(t) \ldp \left( g'(t)/\drift(t)\right)dt
\]
Letting now $T\to \infty$ and then $\eps\to 0$ we get
\[
\limsup_{\lambda\to \infty} \frac{1}{\lambda^2} \log P \Big( \frac{\alpha_\lambda}{\lambda} \in K \Big) \leq -\inf_{g\in K^{3\delta}} \Pathsine(g).
\]
Finally taking $\delta\to 0$ and using the fact that $\Pathsine$ is a good rate function gives the result
\[
\limsup_{\lambda\to \infty} \frac{1}{\lambda^2} \log P \Big( \frac{\alpha_\lambda}{\lambda} \in K \Big) \leq -\inf_{g\in K}\Pathsine(g).
\]
This completes the proof of the lower bound.
\end{proof}
We now prove the  tail bound for the proof of the lower bound. 
\begin{proposition}
\label{tailbound}
Fix $T, \delta>0$, then there is a constant $c>0$ so that 
\begin{align}\label{tail11}
\limsup_{\lambda \to \infty}\frac{1}{\lambda^2} \log P( \alpha_\lambda(\infty)- \alpha_\lambda(T) \geq \delta \lambda) \leq  -c T \delta^2.
\end{align}
\end{proposition}

\begin{proof}[Proof of Proposition \ref{tailbound}]

Take $\nu=1/8$, and set $T_k=\tfrac{k(k+1)}{2}\theta T$ where the value of $\theta>0$ will be specified later. Then we can break up the probability in question as 
\begin{align}\notag
P(\alpha_{\lambda}(\infty)&-\alpha_\lambda(T) \geq \delta \lambda)\le \sum_{k=1}^{\lfloor 2 \sqrt{\lambda} \rfloor } P(\alpha_{\lambda}(T_{k+1})-\alpha_{\lambda}(T_k)\ge \tfrac{\delta}{4} \lambda \nu k^{-(1+\nu)})\\&\hskip150pt+P(\alpha_{\lambda}(\infty)-\alpha_{\lambda}(T_{ \lfloor 2 \sqrt{\lambda}\rfloor +1})\ge \delta \lambda/2).\label{ltail1}
\end{align}
Note, that for any fixed $s>0$ the process $\widehat \alpha_{s,\lambda}(t)=\alpha_\lambda(s+t)$ satisfies the SDE  (\ref{sineSDE2}) with $\hat \lambda=\lambda e^{-\tfrac{\beta}{4}s}$ with initial condition $\alpha_\lambda(s)$. Using the coupling techniques of Propositions \ref{prop:stochdom} and \ref{prop:SDEprop} one can show that $\widehat \alpha_{s,\lambda}(t)-\widehat \alpha_{s,\lambda}(0)=\alpha_\lambda(s+t)-\alpha_\lambda(s)$ is stochastically dominated by $\wt \alpha_{\lambda \drift(s)}+2\pi$. 
This (together with $T_{k+1}-T_k=\theta (k+1)T$) gives
\begin{align*}
 P(\alpha_{\lambda}(T_{k+1})-\alpha_{\lambda}(T_k)\ge \tfrac{\delta}{4} \lambda \nu k^{-(1+\nu)})&\le  P(\wt \alpha_{\lambda \drift(T_k)}(\theta(k+1)T)\ge \tfrac{\delta}{4} \lambda \nu k^{-(1+\nu)}-2\pi)\\
 &\le  P(\wt \alpha_{\lambda \drift(T_k)}(\theta(k+1)T)\ge \tfrac{\delta}{8} \lambda \nu k^{-(1+\nu)})
\end{align*}
where the last bound follows for big enough $\lambda$ from  $k\le 2\sqrt{\lambda}$. 
We can use bound (\ref{uptail}) of Lemma \ref{lem:tailbnd} for the probability on the right with $\tilde \lambda =\lambda \drift(T_k)$, $t=\theta (k+1)T$ and $q=\frac{\delta \nu k^{-(1+\nu)}}{8\theta T (k+1)\drift(T_k)}$, since with these choices $q t \tilde \lambda$, $q$ and $\tilde \lambda q \log q$ are all big, if we choose $\theta>0$ small enough and then  $\lambda$ big enough.
This leads to
\begin{align*}
P(\wt \alpha_{\lambda \drift(T_k)}(T)\ge \tfrac{\delta}{4} \lambda \nu k^{-(1+\nu)})&\le \exp\left(-c_1  \tfrac{\delta^2}{8^2} \lambda^2 \nu^2 k^{-2(1+\nu)}\theta^{-1} (k+1)^{-1}T^{-1} \log^2\left(\frac{\delta \nu k^{-(1+\nu)}}{8\theta (k+1)T \drift(T_k)}   \right)  \right)\\
&\hskip-20pt\le \exp\left(-c_2 \delta^2 \lambda^2 k^{-3-2\nu} T^{-1} \left(c_3+\tfrac{\beta}{4} T \tfrac{k(k+1)}{2}\right)^2\right)\le \exp\left(-c_4 \lambda^2 \delta^2 T k^{1-2\nu}   \right),
\end{align*}
with a positive constant $c_4$, which in turn implies (for large enough $\lambda$) 
\begin{align}
\sum_{k=1}^{\lfloor 2 \sqrt{\lambda} \rfloor } P(\alpha_{\lambda}(T_{k+1})-\alpha_{\lambda}(T_k)\ge \tfrac{\delta}{4} \lambda \nu k^{-(1+\nu)})\le 2 \exp\left(-c_4 \lambda^2 \delta^2 T   \right).\label{sumbnd}
\end{align}
Lastly we bound the remaining term using Proposition \ref{prop:SDEprop}:
\begin{align}
P(\alpha(\infty)- \alpha(T\lambda) \geq \delta \lambda/2)& = P( \xi_{\lambda \drift (\lambda T)}(\infty) \geq \lfloor \delta \lambda/ 2\rfloor) \leq 2 \left( e^{-\frac{\beta}{4} \lambda T}\right)^{\lfloor \delta \lambda/ 2\rfloor},\notag
\end{align}
which together with (\ref{ltail1}) and (\ref{sumbnd}) gives us the necessary upper bound for (\ref{tail11}).
\end{proof}

\begin{proof}[Proof of the lower bound in Theorem \ref{thm:pathsineb}]
%
%
%
%
We will show that if $g\in C[0,\infty)$ with $\Pathsine(g)<\infty$ then 
\begin{align}\label{lowerbndgoal}
\lim\limits_{\eps\to 0} \liminf_{\lambda\to \infty} \tfrac{1}{\lambda^2} \log P(\|\lambda^{-1} \alpha_\lambda(\cdot)-g(\cdot) \|\le \eps)\ge -\Pathsine(g).
\end{align}
From this the lower bound will follow.

In Proposition \ref{prop:goodrate} of the Appendix we will prove that if $\Pathsine(g)<\infty$ then $g(\infty)=\lim\limits_{t\to \infty} g(t)<\infty$ exists. Let $\eps>0$ and choose $T>0$ so that 
\begin{align}
g(\infty)-g(T)\le \eps/2, \qquad \textup{and} \qquad e^{-\tfrac{\beta}{4}T}\le \eps/4.\label{assmpt}
\end{align}
From the first assumption in (\ref{assmpt}) and the Markov property we have
\begin{align}\notag
&P(|\lambda^{-1} \alpha_\lambda(t)-g(t)|\le \eps, t\ge 0)\\
&\hskip20pt\ge P(|\lambda^{-1} \alpha_\lambda(t)-g(t)|\le \eps/2, t\in [0,T],  |\alpha_\lambda(\infty)-\alpha_\lambda(T)|\le \lambda\eps/4)\label{qqq}
\\ &\hskip20pt \ge P(|\lambda^{-1} \alpha_\lambda(t)-g(t)|\le \eps/2, t\in [0,T]) \sup_x P(\alpha_\lambda(\infty)- \alpha_\lambda(T)\le \lambda \eps/4 \big\vert \alpha_\lambda(T)=x).\notag
\end{align}
Using the same line of reasoning as in the proof of Proposition \ref{tailbound} (see after (\ref{ltail1})) we get that with $\lambda_T=\lambda e^{-\tfrac{\beta}{4}T}$ we have
%
%
%
\begin{align*}
&P(\alpha_\lambda(\infty)- \alpha_\lambda(T)\le \lambda \eps/4 \vert \alpha_\lambda(T)=x)\\
&\hskip100pt\ge P(\alpha_{ \lambda_T}(\infty)\le \lambda \eps/4-2\pi)\ge P(\alpha_{ \lambda_T}(\infty)\le \lambda \eps/8),
\end{align*}
where the second inequality follows if  $\lambda$ is big enough compared to $\eps$. Now we can use part (iii) of Proposition \ref{prop:SDEprop} with $f(t)=\lambda_T \drift(t)$, $k=1$ and $a=\lambda \eps/8$  to get
\begin{align*}
P(\alpha_{ \lambda_T}(\infty)\le \lambda \eps/8)=1-P(\alpha_{ \lambda_T}(\infty)> \lambda \eps/8)\ge 1-2\frac{8 \lambda_T}{2\pi \lambda \eps }\ge 1-\frac{2}{\pi},
\end{align*}
where the last step follows from the second assumption of (\ref{assmpt}).

Using this with (\ref{qqq}) we get  that 
\[
\liminf_{\lambda\to \infty} \frac{1}{\lambda^2} \log P(|\lambda^{-1} \alpha_\lambda(t)-g(t)|\le \eps, t\ge 0)\ge \liminf_{\lambda\to \infty} \frac{1}{\lambda^2} \log P(|\lambda^{-1} \alpha_\lambda(t)-g(t)|\le \eps/2, t\in[0,T]),
\]
and it is enough to estimate the right hand side. 
 We do this by introducing the process $\xi_N(t)$ on $[0,T]$ which is a solution of the SDE (\ref{xi}) with initial condition 0 and  the piecewise constant drift function $\lambda \drift_N$ where $\drift_N$ is defined as in (\ref{Pproj}).  From Proposition \ref{prop:SDEprop} we have that $\alpha_\lambda(t)\le \xi_N(t)$ and $\widehat \xi_N(t)=\xi_N(t)-\alpha_\lambda(t)$ satisfies SDE (\ref{xisine2}) with initial condition 0 and drift $\lambda (\drift_N(t)-\drift(t))$.  
We have
\begin{align}\label{term1}
P(|\lambda^{-1} \alpha_\lambda(t)-g(t)|\le \eps/2, t\in [0,T])\ge& P(|\lambda^{-1} \xi_N(t)-g(t)|\le \eps/4, t\in [0,T])\\&\hspace{1cm}-P\bigg(\sup_{t\in [0,T]}|\xi_N(t)- \alpha_\lambda(t)|\ge \lambda\eps/4\bigg).\notag
\end{align}
The second term on the right may be bounded in the same manner as (\ref{error2}).  this gives us
\begin{equation}
\limsup_{\lambda\to \infty} \frac{1}{\lambda^2} \log P\bigg(\sup_{t\in [0,T]}|\xi_N(t)- \alpha_\lambda(t)|\ge \lambda\eps/4\bigg) \leq
-\frac{\beta^2 T^3}{4^2N^2}\ldp \left( \frac{\eps  N}{\beta T^2} \right).\notag
\end{equation}
Note, that as $N\to \infty$ the right hand side converges to $-\infty$.

The only thing left is to estimate the first term on the right of (\ref{term1}). Introduce the notation $t_k=\tfrac{Tk}{N}$. 
We start with the bound
\begin{align*}
&\log P(|\lambda^{-1} \xi_N(t)-g(t)|\le \eps/4, t\in [0,T])\\
&\hskip50pt \ge P(|\lambda^{-1} (\xi_N(s+t_k)-\xi_N(t_k))-(g(s+t_k)-g(t_k))|\le \eps/(4N), s\in [0,T/N]).
\end{align*}
For any fixed $k$ the process $\xi_N(s+t_k), s\in [0,T/N]$ satisfies the SDE (\ref{xisine2}) with initial condition $\xi_N(t_k)$ and a constant drift $\lambda \drift_N(t_k)$. Using the coupling in the proof of Proposition \ref{prop:stochdom} we can construct independent processes $\widehat \alpha_k(t), t\in[0,T/N]$ so that 
\[
\widehat \alpha_k(s)-2\pi \le \xi_N(s+t_k)-\xi_N(t_k)\le \widehat \alpha_k(t)+2\pi, \quad s\in[0,T/N]
\]
and $\widehat \alpha_k(t), t\in[0,T/N]$ has the same distribution as $\wt \alpha_{\lambda \drift_N(t_k)}(t), t\in [0,T/N]$. 
From this it immediately follows that
\begin{align*}
&\liminf_{\lambda\to \infty} \tfrac{1}{\lambda^2} \log P(|\lambda^{-1} \xi_N(t)-g(t)|\le \eps/4, t\in [0,T])\\ 
&\quad \ge \sum_{k=0}^{N-1} \liminf_{\lambda\to \infty} \tfrac{1}{\lambda^2} \log P(|\lambda^{-1} \wt \alpha_{\lambda \drift_N(t_k)}(s)-(g(s+t_k)-g(t_k))|< \eps/(8N), s\in [0,T/N]).
\end{align*}
From our path level large deviation lower bound on $\wt \alpha$ we get
\begin{align*}
 &\liminf_{\lambda\to \infty} \tfrac{1}{\lambda^2} \log P(|\lambda^{-1} \wt \alpha_{\lambda \drift_N(t_k)}(s)-(g(s+t_k)-g(t_k))|\le \eps/(4N), s\in [0,T/N])\\
  &\hskip90pt \ge -\inf_{\substack{|\tilde g(s)-g(s)|<\eps/(8N)\\ s\in[t_k,t_{k+1}]}} \drift_N(t_k)^2 \int_0^{T/N} \ldp(\drift_N(t_k)^{-1} \tilde g'(t_k+s)) ds. 
\end{align*}
This yields the estimate
\begin{align*}
\liminf_{\lambda\to \infty} \tfrac{1}{\lambda^2} \log  P(|\lambda^{-1} \alpha_\lambda(t)-g(t)|\le \eps/4, t\in [0,T])&\ge
-\inf_{\substack{|\tilde g(s)-g(s)|< \eps/(8N),\\ s\in [0,T]}} \int_0^T \drift_N(s)^2 \ldp(\drift_N^{-1} \tilde g'(s)) ds.
\end{align*}
Letting $N\to \infty$ the lower bound converges to $-\int_0^T \drift(s)^2\ldp(\drift^{-1}(s) g'(s))ds$ which (together with our previous estimates) shows that
\[
\liminf_{\lambda\to \infty} \frac{1}{\lambda^2} \log P(|\lambda^{-1} \alpha_\lambda(t)-g(t)|\le \eps, t\ge 0)\ge -\int_0^T \drift(s)^2\ldp(\drift^{-1}(s) g'(s))ds.
\]
Letting $\eps\to 0$ we also have $T=T_\eps\to \infty$ which yields the bound (\ref{lowerbndgoal}) and concludes the proof of the lower bound in the large deviation principle.
\end{proof}


\section{$\Pathschr$ and $\Pathsine$ are good rate functions}\label{s:goodrate}

In this section we will show that $\Pathschr$ and $\Pathsine$ are good rate functions. Our main tools are the bound (\ref{ldptail1}) and the estimate
\begin{align}
\ldp(x)\ge c_1 (x-1)^2, \quad \textup{if $x>0$} \label{quadbnd}
\end{align}
both of which will be proved in Proposition (\ref{prop:ldpasympt}) of the Appendix.

\begin{proposition}\label{prop:goodrate}
The functions $\Pathsine(\cdot)$ and $\Pathschr(\cdot)$ are both good rate functions on the spaces $C[0,\infty)$ and $C[0,T]$ respectively. Moreover, if $g\in C[0,\infty)$ and $\Pathsine(g)<\infty$ then $\lim\limits_{t\to \infty} g(t)$ is finite. 
\end{proposition}
\begin{proof}
Fix $T>0$ and $r\ge 0$.  In order to prove that $K_r=\{g:\Pathschr(\cdot)\le r\}$ is compact we first  show the equicontinuity of this set. Suppose that $g\in K_r$. Then $g(0)=0$ and $g'(x)\ge 0$ exists a.e.~in $[0,T]$. 
We have for $0\le x \le y\le T$
\begin{align*}
|g(x)-g(y)-(x-y) |=\left|\int_x^y (g'(s)-1) ds \right|&\le (y-x)^{1/2} \sqrt{\int_x^y (g'(s)-1)^2 ds}\\
&\le c (y-x)^{1/2}  \sqrt{\int_x^y \ldp(g'(s)) ds}\le c  (y-x)^{1/2} r^{1/2}
\end{align*}
where we used (\ref{quadbnd}) in the second step. This shows that $K_r$ is equicontinuous. Using Tonelli's semicontinuity theorem (e.g.~Theorem 3.5, \cite{variational}) the compactness of $K_r$ now follows. 

The proof for $\Pathsine(\cdot)$ is  bit more involved. Fix $\beta>0$. It is convenient to transform the interval $[0,\infty)$ into $[0,1)$ using the function $y=1-e^{-\beta t/4}$. Then for a  $g\in C[0,\infty)$ with $\Pathsine(g)<\infty$ we have
\[
\Pathsine(g)=\int_0^\infty \drift^2(t) \ldp(g'(t) \drift^{-1}(t)) dt=\frac{\beta}{4} \int_0^1(1-y)\ldp(\tilde g'(y)) dy 
\]
where $\tilde g(y)=g(-\tfrac4{\beta} \log (1-y))$, $\tilde g\in C[0,1)$. Consider the functional  $\Psine(\cdot)$ on $C[0,1)$ defined as 
\begin{align}\label{Psine}
\Psine(g)=
\tfrac{\beta}{4}\int_0^1 (1-t) \ldp(g'(t))dt
\end{align}
if  $g'(t)$ exists and non-negative for a.e. $0\le t<1$, and as $\infty$ otherwise. Clearly, if we show that $\Psine(\cdot)$ 
is a good rate function on $C[0,1)$ then the same will hold for $\Pathsine$.  We first show that if $g\in C[0,1)$ and $\Psine(g)<\infty$ then $\lim\limits_{y\to1^-}g(y)$ is finite, i.e.~we can consider $\Psine(\cdot)$ on  $C[0,1]$. We have
\begin{align*}
\lim\limits_{y\to 1^{-}} g(y)=\int_0^1 g'(y) dy\le 2+ \int_0^1 g'(y)\ind(g'(y)\ge 2) dy.
\end{align*}
We will prove that 
\begin{align}
\textup{if }h(y)\ge 0,\,\,\textup{ and } \,\,\int_0^1 (1-y) h(y)^2 \log^2(h(y)+e)dy<\infty,\,\, \textup{ then } \,\,\int_0^1 h(y) dy<\infty.
\label{BBB}
\end{align}
Using this with $h(y)=g'(y)\ind(g'(y)\ge 2)$ together with the  bound in (\ref{ldptail1}) we get the boundedness of $\int_0^1 g'(y) dy$ and the existence on $\lim\limits_{y\to1^{-}} g(y)$. 

Let $\Phi(x)=x^2 \log^2(|x|+e)$, this is a strictly convex,  even function with $\lim\limits_{x\to 0} \frac{\Phi(x)}{x}=0$, $\lim\limits_{x\to \infty} \frac{\Phi(x)}{x}=\infty$ (i.e.~$\Phi$ is a `nice Young function').  Introduce the complementary function
\[
\Psi(x)=\Phi^*(x)=\sup_{y\ge 0} \{ y |x|-\Phi(y)\}=\int_0^{|x|} (\Phi')^{(-1)}(y) dy
\] 
where $(\Phi')^{(-1)}$ is the inverse of the strictly increasing function $\Phi'$ on $[0,\infty)$. Assume that 
\begin{align}\label{AAA}
A=\int_0^1 (1-y) \Phi(h(y)) dy<\infty
\end{align}
and let $\mu$ the measure on $[0,1]$ with $d\mu =\frac{1}{A}(1-x) dx$. Consider the Orlicz spaces
\begin{align*}
L^\Phi_\mu&=\{f: \textup{there is an $a>0$ with } \int_{[0,1]} \Phi(a f) d\mu<\infty \},\\
L^\Psi_\mu&=\{f: \textup{there is an $a>0$ with } \int_{[0,1]} \Psi(a f) d\mu<\infty \}
\end{align*}
with the Luxemburg-norms defined as
\begin{align}\label{norm}
\| f\|_{\Phi}=\inf\{b>0: \int_{[0,1]} \Phi(b^{-1} f) d\mu \le 1\}, \quad \| f\|_{\Psi}=\inf\{b>0: \int_{[0,1]} \Psi(b^{-1} f) d\mu \le 1\}.
\end{align}
(See e.g.~\cite{Orlicz} for more on Orlicz spaces.)
Note, that by our assumption (\ref{AAA}) we have $\| h\|_\Phi \le 1$. By the generalized H\"older inequality for Orlicz spaces (c.f.~Theorem 3 in Chapter III of \cite{Orlicz}), for any $ f\in L^\Psi_\mu$ one has
\begin{align}\label{orlicz}
\| f h\|_1 \le 2 \|f\|_\Psi \|h\|_\Phi\le 2 \|f\|_\Psi
\end{align}
where $\|\cdot \|_1$ is the $L^1$ norm on $[0,1]$ with reference measure $\mu$. 
Choose $f(x)=\frac{1}{1-x}$. If we show that $\|f\|_\Psi<\infty$ then this would imply
\[
\infty>2\|f\|_\Psi\ge  \| f h\|_1= \tfrac{1}{A} \int_0^1  \frac{1}{1-x} h(x) (1-x)dx=\tfrac1{A} \int_0^1 h(x) dx,
\]
and  the statement (\ref{BBB}) would follow. It is not hard to check, that there is a $c>0$ so that 
\begin{align}\label{Psi}
\Psi(x)\le c \frac{x^2}{\log^2(x+e)}, \qquad \textup{for $x\ge 0$.}
\end{align}
Since the integral  $\int_0^1 (1-x) \frac{(1-x)^{-2}}{\log^2((1-x)^{-1}+e)}dx$ is finite, this implies that  $\|\tfrac{1}{1-x}\|_{\Psi}$ is finite and thus $\int_0^1 h(y) dy<\infty$. This completes the proof that if
$\Psine(g)<\infty$ then $\lim\limits_{y\to1^-}g(y)$ is finite, and also shows the last statement of the proposition. 

Next we will prove the equicontinuity of the set $K_r=\{f: \Psine(f)\le r\}$, we will show that if $g\in K_r$ then for   $\eps<\eps_0$  we have
\begin{align}\label{equicont}
|g(a+\eps)-g(a)|\le C (\log \eps^{-1})^{-1/3} \qquad \textup{for any $a\in [0,1-\eps]$}.
\end{align}
Here $\eps_0, C$ only depend on $r$. 

We first assume $a\le 1-\sqrt{\eps}$. Then
\begin{align*}
|g(a+\eps)-g(a)-\eps|&=|\int_a^{a+\eps} ( g'(y)-1) dy|
\\
&\le \left(\int_a^{a+\eps} \frac{1}{1-y}dy\right)^{1/2} \left(\int_a^{a+\eps} (1-y) (g'(y)-1)^2 dy\right)^{1/2} \\
&\le C r^{1/2} \left(\log\left(1+\frac{\eps}{1-a-\eps}\right)\right)^{1/2}\le C r^{1/2} \eps^{1/4}
\end{align*}
Where we used $1-a>\sqrt{\eps}$, the bound (\ref{quadbnd}) and the fact that $\eps$ can be chosen to be small enough. 

Next we assume that $a>1-\sqrt{\eps}$. Because of the monotonicity of $g$ it is enough to bound $|g(1)-g(1-\sqrt{\eps})|$. Setting  $f(x)=\tfrac{1}{1-x}$ and $h(x)=g'(x)\ind(g'(x)\ge 2)$ we  have 
\begin{align}\label{ppp}
g(1)-g(1-\sqrt{\eps})\le 2\sqrt{\eps}+\int_{1-\sqrt{\eps}}^1 h(x) dx. 
\end{align}
Since $\Psine(g)\le r$, we can assume that (\ref{AAA}) holds with some finite $A>0$. We will now follow the previous argument using  Orlicz spaces. We use the same definitions for $\Psi, \Phi$, $\mu$ but for the norms $\|\cdot \|_\Psi, \|\cdot \|_\Phi$ defined in (\ref{norm}) we use the interval  $[1-\sqrt{\eps}, 1]$ instead of $[0,1]$.

Using inequality (\ref{ppp}) and  (\ref{orlicz})  we get the bound 
\[
g(1)-g(1-\sqrt{\eps})\le 2\sqrt{\eps}+A\int_{1-\sqrt{\eps}}^1 f(x) h(x) d\mu(x)\le 2\sqrt{\eps}+  A\|f\|_\Psi.
\]
To estimate $ \|f\|_\Psi$ we will prove that with  $b=(\log \eps^{-1})^{-1/3}$ there is a constant $\eps_0$ depending on $A$ so that 
\[
\int_{1-\sqrt{\eps}}^1 \Psi(b^{-1} f(x)) d\mu(x)=A^{-1}\int_{1-\sqrt{\eps}}^1(1-x) \Psi(b^{-1} (1-x)^{-1}) dx<1, \qquad \textup{for $\eps<\eps_0$}.
\]
This will imply that for such $\eps$ we have $\|f\|_\Psi\le b$.   Using (\ref{Psi}) we get
\begin{align*}
A^{-1} \int_{1-\sqrt{\eps}}^1(1-x) \Psi(b^{-1} (1-x)^{-1}) dx&\le c b^{-2}A^{-1} \int_0^{\sqrt{\eps}} x^{-1} \frac{1}{\log^2(2x^{-1} b^{-1})}dx\\
&\le \frac{c (\log \eps^{-1})^{2/3}}{A \log\left(2 (\log \eps^{-1})^{-1/3} \eps^{-1/2}\right)}. 
\end{align*}
Since the right hand side converges to 0 as $\eps\to 0$ we get that  $\|f\|_\Psi\le b$ for small enough $\eps$  which in turn  leads to the upper bound 
(\ref{equicont}). This completes the proof of the equicontinuity of the set $K_r$ and the compactness follows again by Tonelli's theorem.
\end{proof}


\section{From the path to the endpoint}\label{s:endpoint}

In this section we will complete the proofs of Theorems \ref{thm:sineb1} and \ref{thm:schr1}.

\begin{proof}[Proof of Theorem \ref{thm:schr1}]

%
%
%

Consider the continuous map $F:C[0,T]\to \R$ given by $F(g)=g(T)/(2\pi)$.  By the contraction principle (see e.g.  \cite{DemboZeitouni}) the random variables $\tfrac1{\lambda}\tfrac{\alpha_\lambda(T)}{2\pi}$  satisfy a large deviation principle with scale function $\lambda^2$ and good rate function $J$ defined as
\begin{equation}
\label{neardecreasingg}
J(\rho) = \min\left\{ \int_0^T \ldp(g'(t))dt: \, g'(t)\ge 0, \,g(T)=2\pi \rho\right\}. 
\end{equation}
We will now solve this variational problem. If $g$ provides the minimum then we can assume that $g'$ is monotone decreasing.  To see this define $\tilde g$ with $\tilde g(0)=0$ and  $\tilde g'(t) = \sup\{x : m(g'(s)\geq x)\geq t\}$ where $m$ indicates Lesbegue measure. Then $\tilde g(T)=g(T)$, $\Pathschr (\tilde g)= \Pathschr (g)$, and $\tilde g'(t)$ is decreasing.  If $g'>0$ on $[0,a]$ and $g'=0$ on $(a,1]$ then by the classical variational method we get that $\ldp'(g'(x))$ is constant  on $[0,a]$. This means that $g'(x)=\frac{2\pi\rho}{a}$ on $[0,a]$ and $g'(x)=0$ on $(a,T]$ and our variational problem is reduced to finding the minimum of 
\begin{align*}
f(a)=a \ldp\left(\tfrac{2\pi \rho}{a}\right)+(T-a) \ldp(0), \quad \textup{on } \quad 0\le a\le T.
\end{align*}
But we have 
\[
f'(a)=\ldp\left(\tfrac{2\pi\rho}{a}\right)-\ldp(0)-\ldp'\left(\tfrac{2\pi\rho}{a}\right)\tfrac{2\pi \rho}{a}<0,
\]
since $\ldp$ is strictly convex,  which means that the minimum is at $a=T$. Thus
\begin{align}\label{ldprate11}
J(\rho)=\min\left\{ \int_0^T \ldp(g'(t))dt, \, g'(t)\ge 0, \,g(T)=2\pi \rho\right\}=T \ldp(2\pi \rho/T) 
\end{align}
is the large deviation rate function for $\frac{1}{\lambda}\frac{\alpha_\lambda(T)}{2\pi}$.  

Now recall that the counting function of $\sch_\tau$ is given by
\[
\wt N_\tau(\lambda)=\# \{ \nu: 0\le \nu \le \lambda, \phi_{\lambda/\tau}(\tau)\in 2\pi \Z\}
\]
 where $\phi_\lambda$ is the solution of (\ref{schrSDE1}). Note, that $\phi_\lambda(t)-\phi_0(t)$ has the same distribution as $\xi_{f,0}(t)$ with constant $f=\lambda$, which in turn has the same distribution as $\wt \alpha_\lambda(t)$. Using the coupling methods of Proposition \ref{prop:SDEprop} we can show that $\phi_\lambda(t)$ is increasing in $\lambda$ for any fixed $t$ (see \cite{KVV} for a detailed proof of this fact). From this it follows that
\[
\left|\tilde N_\tau(\lambda)-\tfrac{1}{2\pi}\left(\phi_{\lambda/\tau}(\tau)-\phi_0(\tau)\right)\right|\le 1.
\]
This means that in order to get a large deviation principle for $\tfrac1{\lambda}\tilde N_\tau(\lambda)$ it is enough to prove one for $\tfrac1{\lambda} \tfrac{\phi_{\lambda/\tau}(\tau)-\phi_0(\tau)}{2\pi}$. But this has the same distribution as $\frac{1}{\lambda}\frac{\wt \alpha_{\lambda/\tau}(\tau)}{2\pi}$, and a simple rescaling of (\ref{ldprate11}) completes the proof of the theorem.
\end{proof}

\begin{proof}[Proof of Theorem \ref{thm:sineb1}]
Theorem \ref{thm:pathsineb} shows that $\tfrac1{\lambda} \alpha_{\lambda(\cdot)}$ satisfies a path level large deviation principle. By applying the time change $y=1- e^{-\tfrac{\beta}{4}t}$, we get that $t\to \tfrac1{\lambda} \alpha_{\lambda}(1- e^{-\tfrac{\beta}{4}t})$ satisfies a path level LDP on $C[0,1)$ with the modified rate function $\Psine$ given in (\ref{Psine}). In Proposition \ref{prop:goodrate} we showed  that if $\Psine (g)<\infty$ then the limit as $t\to 1^{-}$ exists and so the LDP actually holds on $C[0,1]$. Using the contraction principle with the functional $F(g)= \tfrac{1}{2\pi} g(1)$, we get that $\tfrac{1}{\lambda}\tfrac{\alpha_\lambda(\infty)}{2\pi}$ satisfies a large deviation principle with speed function $\lambda^2$ and a good rate function 
\begin{align*}
J^\beta(\rho)& = \min\left\{\Psine(g): \, g(1)=2\pi \rho\right\}\\&= \min\left\{ \tfrac{\beta}{4}\int_0^1 (1-t) \ldp(g'(t))dt: \, g(0)=0, g'(t)\ge 0, g(1)=2\pi \rho\right\}. 
\end{align*}
 The counting function $N_\beta(\lambda)$ of $\Sineb$ is given by $\tfrac{\alpha_\lambda(\infty)}{2\pi}$, so Theorem \ref{thm:sineb1} will follow if we can show that the solution of this variational problem is given by $\beta \sineldp(\rho)$ as defined in the theorem.

The function $\Psine$ is a good rate function, so for any $\rho\ge 0$ the minimum is achieved at some $g_\rho\in C[0,1]$.
Clearly, when $\rho=\tfrac{1}{2\pi}$ then the minimum is zero, as the $g(t)=t$ function shows. (We will not denote the dependence of $\rho$ in $g=g_\rho$ from this point.)

We may assume that for the minimizer the derivative $g'$ will not take values from both $(1,\infty)$ and $[0,1)$ because otherwise we could construct a function $\hat g$ with the same boundary condition $\hat g(1)=2\pi \rho$, but with $\Psine(g) >\Psine (\hat g)$.  The construction is as follows.  Assume $\rho <1/(2\pi)$ and that  $A= \{t: g'(t) >1\}$ has positive measure. Since $\int_0^1 (g'(t)-1) dt=2\pi \rho-1<0$ and $\int_A (g'(t)-1) dt>0$, by the intermediate value theorem we can find $B\subset [0,1]\setminus A$ so that $\int_{A\cup B} (g'(t)-1)dt=0$. Define $\hat g$ with $\hat g(0)=0$, 
$\hat g'(t)= g'(t)$ if $t\notin A\cup B$ and $\hat g'(t) =1$ otherwise. Then $g(1)=\int_0^1 g'(t) dt=\int_0^1 \hat g'(t) dt=\hat g(1)$, but clearly  $\Psine(g) > \Psine (\hat g)$. 
  A similar construction works for $\rho > 1/(2\pi)$.  Thus we may assume that $g'(t)\le 1$ for all $t$  if $\rho< 1/(2\pi)$, and $g'(t)\ge 1$ for all $t$ if $\rho>1/(2\pi)$.

First assume that $\rho>\tfrac{1}{2\pi}$. Then $g'(t)\ge 1$ for all $t$ and we can use the classical variational method (see e.g.~\cite{variational}) to conclude that $(1-t)\ldp'(g'(t))$ is constant in $t$. Thus the optimizer is given by a function $g_\rho$ which satisfies
\begin{equation}
\label{grho1}
g_\rho(0)=0, \qquad \ldp'(g'(t))=\tfrac{c_\rho}{1-t}, \qquad \int_0^1 (\ldp')^{(-1)}\left(\tfrac{c_\rho}{1-t}\right) dt=2\pi \rho,
\end{equation}
for some constant $c_\rho$ and the solution of the variational problem is 
\begin{align}\label{Jb1}
J^\beta(\rho)=\tfrac{\beta}{4}\int_0^1 (1-t) \ldp\left((\ldp')^{(-1)}\left(\tfrac{c_\rho}{1-t}\right)\right) dt. 
\end{align}
In Proposition \ref{prop:simplifiedRF} below we will show that this is equal to $\beta I_{Sine}(\rho)$ as defined in Theorem \ref{thm:sineb1}.

Now assume that $\rho<\tfrac{1}{2\pi}$, here we can assume that the minimizer satisfies $g'(t)\le 1$. As in the case of $\sch_\tau$ we may assume $g'$ is decreasing, this can be shown using the same construction as found in the paragraph directly following equation (\ref{neardecreasingg}).  Suppose that $g'$  is zero for $t\in [a,1]$ and $g'(t)>0$ in $[0,a]$. Then on $[0,a]$ the classical variational method shows that $(1-t)\ldp'(g'(t))$ must be constant. Thus the optimizer must be of the following form:
\begin{equation}
\label{grho2}
g'(t)=\begin{cases}
(\ldp')^{(-1)}\left(\tfrac{c_{\rho,a}}{1-t}\right), &\qquad 0\le t \le a\\
0, &\qquad a<t\le 1,
\end{cases}
\end{equation}
for some constant $c_\rho$ which satisfies
\begin{align}\label{defc}
2\pi \rho =\int_0^a (\ldp')^{(-1)}\left(\tfrac{c_{\rho,a}}{1-t}\right) dt.
\end{align}
By Propositions \ref{prop:ldpconvex} and \ref{prop:ldpasympt} of the Appendix the function $\ldp'(x)$ is strictly increasing on $(0,\infty)$ with a limit of $-\tfrac{1}{2\pi}$ at $x=0$. Thus $c_{\rho,a}$ in (\ref{grho2}) cannot be smaller than $-\frac{1-a}{2\pi}$. Our next claim is that the optimizer has a continuous derivative at $t=a$, which will identify $c_{\rho,a}$ as $c=-\tfrac{1-a}{2\pi}$. 
Assume the opposite, i.e.~that $c_{\rho,a}>-\tfrac{1-a}{2\pi}$ and $g'(a)>0$. Let $\eta_{\delta}(x)=\ind_{(a,a+\delta)}-\ind_{(a-\delta,a)}$. If $\delta, \eps$ are small enough then $g'-\eps \eta_{\delta}\ge 0$ in $[0,1]$ and $\tilde g(t)=\int_0^t (g'(s)-\eps \eta_{\delta}(s))ds$ satisfies the same boundary conditions as $g$. Since $g$ is a minimizer, the derivative of  $h(\eps)= \Psine(g+\eps \eta_\delta)$ at $\eps=0$ cannot be negative. We can compute the derivative as
\begin{align*}
h'(0)=\frac{\beta}{4} \int_0^1 (1-t)\ldp'(g(t)) \eta_\delta(t)dt=-\int_{a-\delta}^a (1-t)\tfrac{c_{\rho,a}}{1-t}dt+\int_a^{a+\delta} (1-t) \left(-\tfrac1{2\pi}\right)dt.
\end{align*}
This is equal to $\delta(-c_{\rho,a}-\tfrac{1-a}{2\pi})+\tfrac{\delta^2}{4\pi}$ which is negative if $\delta$ is small enough (by our assumption that $c_{\rho,a}>-\tfrac{1-a}{2\pi}$). The contradiction shows that we must have  $c=-\tfrac{1-a}{2\pi}$. 

Thus the optimizer is given by 
\begin{equation}
\label{grho3}
g'(t)=\begin{cases}
(\ldp')^{(-1)}\left(\tfrac{a-1}{2\pi(1-t)}\right), &\qquad 0\le t \le a\\
0, &\qquad a<t\le 1.
\end{cases}
\end{equation}
for some $0\le a\le 1$ with
\begin{align}\label{defc1}
2\pi \rho =\int_0^a (\ldp')^{(-1)}\left(\tfrac{a-1}{2\pi(1-t)}\right) dt.
\end{align}
and the solution of the variational problem in the $2\pi \rho<1$ case is given by 
\begin{align}\label{Jb2}
J^\beta(\rho)=\frac{\beta}{4} \int_0^a (1-t)  \ldp\left((\ldp')^{(-1)}\left(\tfrac{a-1}{2\pi(1-t)}\right)\right) dt+\tfrac{\beta}{64}(1-a)^2.
\end{align}
In Proposition \ref{prop:simplifiedRF} below we will show that this is equal to $\beta \sineldp(\rho)$. 
%
%
\end{proof}

\begin{proposition}\label{prop:simplifiedRF}
The rate function for the $\Sineb$ process is given by 
\begin{align}
\beta \sineldp (\rho)=\frac{\beta}{8} \left[ \frac{\nu}{8}+ \rho \II(\nu)\right]\notag
\end{align}
where $\nu=\gamma^{-1}(\rho)$, and $\gamma$ is the strictly increasing function given in (\ref{rhonu}).
\end{proposition}

\begin{proof}
We have to show that $J^\beta(\rho)$ defined by (\ref{grho1}) and (\ref{Jb1}) for $\rho>1/(2\pi)$ and by (\ref{defc1}) and (\ref{Jb2}) for $\rho<1/(2\pi)$ is equal to $\beta I_{Sine}$ given above.

We begin with the case where $\rho>\tfrac{1}{2\pi}$.  In this case the minimizer $g=g_\rho$ is given by (\ref{grho1}). One easily checks that 
\begin{align}\label{exact}
\frac{d}{dt}\left(\tfrac{\beta}{8}\left(-(1-t)^2 \ldp(g'(t))+c_\rho (1-t)g'(t)+c_\rho g(t)\right)\right)=\tfrac{\beta}{4}(1-t)\ldp(g'(t)).
\end{align}
From this we get
\[
J^\beta(\rho)=\frac{\beta}{4} \int_0^1 (1-t)  \ldp \left( g'(t))  \right)dt= \tfrac{\beta}{8} \left[ \ldp(g'(0))- c_\rho g'(0)+2 \pi \rho c_\rho\right]
\]
where we used $g(0)=0$, $g(1)=2\pi \rho$, and the limits
\[
\lim\limits_{t\to 1^-} (1-t)^2 \ldp(g'(t))=\lim\limits_{x\to \infty} \frac{c_\rho^2 \ldp(x)}{\ldp'(x)^2}=0, \qquad \lim\limits_{t\to 1^{-}} (1-t) g'(t)=\lim\limits_{x\to \infty} \frac{x}{c_\rho \ldp'(x)}=0 
\]
which follow from the asymptotics (\ref{limits}) and (\ref{lim_ldp'}) to be proven in Proposition \ref{prop:ldpasympt}.

Now for the case where $\rho<\tfrac{1}{2\pi}$ we have that $g_\rho$ is given by (\ref{grho3}).  Using the notation $c=c_\rho=\frac{a-1}{2\pi}$, the identity (\ref{exact}) gives 
\[
J^\beta(\rho)=\frac{\beta}{4} \int_0^{2\pi c_\rho+1} (1-t)  \ldp \left( g'(t))  \right)dt+ \tfrac{\beta}{8}(2\pi c_\rho)^2\ldp(0)=  \tfrac{\beta}{8} \left[ \ldp(g'(0))- c_\rho g'(0)+2 \pi \rho c_\rho\right],
\]
where we used $g(0)=0$, $g(a)=2\pi \rho$, and $g'(a)=0$.
Note, that $c_\rho>0$ if $\rho>1/(2\pi)$ and $-\tfrac{1}{2\pi}\le c_\rho<0$ if $\rho<1/(2\pi)$. 
Introducing  $\nu=K^{(-1)}\left(\frac{\pi}{2 (\ldp')^{(-1)}(c_\rho)}\right)$, we get for both $\rho<1/(2\pi)$ and $\rho>1/(2\pi)$ that
\[
J^\beta(\rho)=\frac{\beta}{8}\left(\frac{\nu}{8}+\rho \II(\nu)\right)
\]
which agrees with (\ref{defsineldp}), we just have to show that $\nu=\gamma^{(-1)}(\rho)$. Note, that $\nu=\nu(\rho)<0$ if $2\pi \rho>1$ and $0<\nu<1$ if $2\pi \rho<1$.

Recall from (\ref{grho1}) and (\ref{defc1}) that 
\[
\rho=  \tfrac{1}{2\pi}\int_{0}^1 (\ldp')^{(-1)}\left( \tfrac{c_\rho}{1-t}\right) dt, \hspace{.3cm} \text{ if } \rho>\frac{1}{2\pi}, \hspace{.5cm} \text{ and } \hspace{.5cm} \rho= \tfrac{1}{2\pi} \int_{0}^{2\pi c+1} (\ldp')^{(-1)}\left( \tfrac{c_\rho}{1-t}\right) dt, \hspace{.3cm} \text{ if } \rho< \frac{1}{2\pi}.
\]
Applying the change of variables to both integrals with a new variable $x$ satisfying $\frac{\pi}{2 K(x)}=(\ldp')^{(-1)}(\tfrac{c_\rho}{1-t})$, we get that $\rho$ depends on $\nu=K^{(-1)}\left(\frac{\pi}{2 (\ldp')^{(-1)}(c_\rho)}\right)$ exactly via (\ref{rhonu}) which finishes the proof. Note, that the finiteness of the integrals in (\ref{rhonu}) follow from the asymptotics of $K(x)$ and $E(x)$ near 1 and $-\infty$ (see the proof of Proposition \ref{prop:ldpasympt}). 
\end{proof}


\section*{Appendix A: All about $\ldp$}\label{s:app}

%
%
%
%
%


In this section we prove the needed estimates about the function $\ldp$.

\begin{proposition}\label{prop:ldpconvex}
The function $\ldp(x)$ is strictly convex and continuous on $(0,\infty)$. It has an absolute minimum at $x=1$ where it is equal to 0. 
\end{proposition}
\begin{proof}
$K$ is strictly increasing on $(-\infty,1)$ which shows that $\ldp(x)$ is well-defined (and differentiable) on $(0,\infty)$. 
By differentiating (\ref{ldpdef}) and using the identities
\[
K'(x)=\frac{E(x)-(1-x)K(x)}{2(1-x)x}, \qquad E'(x)=\frac{E(x)-K(x)}{2x}
\] 
 we can compute that
\begin{equation}
\label{ldp'}
\ldp'(x)=\frac{1}{x}\ldp (x)- \frac{1}{8 x}K^{-1}\left( \frac{\pi}{2x}\right),
\end{equation}
and
\begin{equation}
\label{I''}
\ldp''(x)=\frac{\pi}{16  x^3} \frac{1}{K'(K^{-1}(\frac{\pi}{2x}))}.
\end{equation}
Observe that $K'(y)> 0$ for $y<1$ which gives $\ldp''(x)> 0$ for $x>0$ and the strict convexity of  $\ldp$. 

Using $K(0)=E(0)=\tfrac{\pi}{2}$ we get $\ldp(1)=\ldp'(1)=0$ which (by the strict convexity) proves the second half of the proposition.
\end{proof}

\begin{proposition}\label{prop:ldpasympt}
We have  $\lim\limits_{x\to0^{+}}\ldp(x)=\tfrac18$ and $ \lim\limits_{x\to0^{+}}\ldp'(x)=-\tfrac1{2\pi}$. There is a constant $c_1>0$ so that \begin{align}\label{ldp_quad}
\ldp(x)&\ge c_1 (x-1)^2, \qquad \textup{ for all $x$}, \textup{ and }
\end{align}
\begin{align}
 \frac{\II(-x)}{\sqrt{x} \log x}, \, \frac{\ldp(x)}{x^2\log^2 x} \textup{ and }\frac{-K^{-1}(1/x)}{x^2 \log^2 x} \textup{ are bounded away from 0 and $\infty$ for $x>2$.}
%
%
%
 \label{ldp_log}
\end{align}
\end{proposition}
\begin{proof}
The following asymptotics can be readily derived from the definitions of elliptic integrals (or by the existing more sophisticated expansions c.f.~\cite{gustafson1},\cite{gustafson2}). There is a constant $c>0$ so that 
\begin{align}\label{elliptic1}
\left|K(-a)-\frac{1}{2\sqrt{a}}\log(16 a)\right|\le \frac{c}{a^{3/2}}\log(a), \quad \left|E(-a)-\sqrt{a} \right|\le \frac{c}{a^{1/2}}\log(a), \quad \textup{for } a>2. 
\end{align}
From this it is easy to check that 
\begin{align}
\lim\limits_{x\to \infty} \frac{\II(-x)}{\sqrt{x} \log x}=\frac12, \qquad
\lim\limits_{x\to \infty} \frac{-K^{-1}(1/x)}{x^2\log^2 x}=1, \quad \textup{and} \qquad
\lim\limits_{x\to \infty} \frac{\ldp(x)}{ x^2 \log^2 x}=\frac{1}{2\pi^2}. \label{limits}
\end{align}
This gives (\ref{ldp_log}). Note, that together with (\ref{ldp'}) this also gives
\begin{align}\label{lim_ldp'}
\lim\limits_{x\to \infty} \frac{\ldp'(x)}{x\log^2 x}=\frac{1}{\pi^2}.
\end{align}
Using the functional identities
\[
E(z)=\sqrt{1-z} E\left(\frac{z}{z-1}\right), \qquad K(z)=\frac{1}{\sqrt{1-z}} K\left(\frac{z}{z-1}\right), \qquad z\in (0,1)
\]
the asymptotics of  (\ref{elliptic1})  can be transformed into 
\[
K(a)\sim -\tfrac{1}{2}\log(1-a), \qquad E(a)\sim 1, \qquad \textup{ as }a\to 1^{-1}
\]
with explicit error bounds for $2/3<a<1$. 
From this we can obtain $\lim\limits_{x\to 0^{+}} \ldp(x)=\tfrac18$ and $ \lim\limits_{x\to0^{+}}\ldp'(x)=-\tfrac1{2\pi}$. Using (\ref{ldp_log}) with the continuity of $\ldp$ and the fact that $\ldp(1)=0$ is an absolute minimum with $\ldp''(1)>0$ gives (\ref{ldp_quad}).
\end{proof}

The following two lemmas help to consolidate error terms that appear in the proofs Theorems \ref{thm:sineb1} and \ref{thm:schr1}.

\begin{lemma}
\label{Hbound}
There exists an absolute constant $c$ such that for any $t, q>0$ we have 
\begin{equation}
|\II(a)|+|a|t/2 \leq c(t+1)( \mathcal{I}(q)+1)
\end{equation}
where $a=a(q)=K^{-1}(\pi/(2q))$. 
\end{lemma}
\begin{proof}
Using   (\ref{elliptic1}) with the definition (\ref{ldpdef}) we get that there is a constant $c_2$ so that 
\begin{align}
c_2^{-1} a(q)\le \ldp(q)\le c_2 a(q), \qquad \textup{if $q>2$},
\end{align}
and the same bounds also give
\begin{align}
|\II(a(q))|\le c_3 \sqrt{|a(q)|} \log |a(q)|
\end{align}
for some constant $c_3$ in the same region. This shows the existence of a constant $A$ with
\[ 
|\II(a)| \leq A \, \mathcal{I}(q), \hspace{1cm} \text{ and } \hspace{1cm} a(q) \leq A  \mathcal{I}(q), \qquad \textup{for $q>2$}.
\]
Since for  $0<q<2$ both $a(q)$ and $\II(a(q))$ are bounded  the lemma  follows. 
%
%
%
%
\end{proof}

\begin{lemma}
\label{Iexpansion}
For any $0\leq \eps < 1/2$ there exists an absolute constant $c$, so that 
\begin{equation}
\mathcal{I}(x+\varepsilon) \leq (1+\varepsilon)I(x)+ c \varepsilon\label{xxx5}
\end{equation}
\end{lemma}

\begin{proof}
Since $\ldp $ is convex, we have
\[
\mathcal{I}(x+ \varepsilon)\leq \mathcal{I}(x)+ \varepsilon \mathcal{I}'(x+\epsilon).
\]
Since $\ldp(x)$ is decreasing on $[0,\pi/2]$, the bound (\ref{xxx5}) follows immediately for $x\in [0,\pi/2-\eps]$ with any $c\ge 0$.  Using (\ref{ldp'}) we  get
\begin{align}
\mathcal{I}(x+ \varepsilon)&\leq \mathcal{I}(x)+ \frac{\varepsilon}{x+\varepsilon}\left(\mathcal{I}(x+\eps)- \frac18 K^{-1}\left(\frac{\pi}{2(x+\varepsilon)}
\right)\right).
\end{align}
From (\ref{limits})  it follows that there  exists an $x_0>0$ such that 
\[
\frac{\varepsilon}{x+\varepsilon}\left(\mathcal{I}(x+\eps)- \frac18 K^{-1}\left(\frac{\pi}{2(x+\varepsilon)}\right)\right)\le \eps \ldp(x), \qquad \textup{if $x\ge x_0$}
\]
uniformly in $\eps\in [0,1/2]$.
Therefore, for $x>x_0$ we have that $\mathcal{I}(x+\varepsilon) \leq (1+\varepsilon)\mathcal{I}(x).$
We can assume $x_0>\pi/2$. By choosing 
\[
c= \sup_{x\in [\pi/2, x_0+1/2]} \ldp'(x)=\ldp'(x_0+1/2)
\]
we get $\mathcal{I}(x+\varepsilon) \leq \mathcal{I}(x)+c \varepsilon$ on $[\pi/2-\eps, x_0]$ with any $0\le \eps<1/2$ and the lemma follows.
\end{proof}

\section*{Appendix B: Properties of $\sineldp$}

In the final section of the appendix we describe the  behavior of the function $\sineldp(\rho)$ near $\rho=\tfrac{1}{2\pi}$ and $\rho \to \infty$. 

\begin{proposition}\label{prop:sineldp}
The functions $\gamma(\nu)$  and  $\sineldp(\rho)$ satisfy the following.
\begin{enumerate}
\item The function $\gamma(\nu)$ defined in (\ref{rhonu}) is  continuous and strictly decreasing. 
It  satisfies the differential equation $4x(1-x)\gamma''(x)=\gamma(x)$ on $(-\infty,0)$ and on  $(0,1)$ with boundary behavior $\lim\limits_{x\to 0^{\pm}} \gamma(x)=\tfrac{1}{2\pi}$, $ \gamma(1)=0$ and $\lim\limits_{x\to -\infty} \frac{\gamma(x)}{\sqrt{|x|}}=\tfrac14$.
These limits and the differential equation identify $\gamma(x)$ uniquely on $(-\infty,0)\cup(0,1]$.

\item  
%
%
%
%

We have $\sineldp(0)=\tfrac{1}{64}$, $\sineldp(\tfrac{1}{2\pi})=0$, and  $\sineldp''(x)>0$ for $x\neq \tfrac{1}{2\pi}$. Moreover, we have the following limits:
\[
\lim\limits_{x\to 0} \frac{\sineldp (\tfrac{1}{2\pi}+x)}{\tfrac{x^2}{\log(1/ |x|)}}=\frac{\pi^2}{4} , \qquad \textup{and} \qquad \lim\limits_{\rho\to \infty} \frac{\sineldp (\rho)}{\rho^2\log \rho}=\frac12.
\]

%
%
%
\end{enumerate}
\end{proposition}

\begin{proof}
Recall the function $\gamma(\nu)$ given in  (\ref{rhonu}). Using the asymptotics (\ref{limits}) proved in Proposition \ref{prop:ldpasympt} it is easy to see that $\gamma(\nu)$ is well-defined and positive in $(-\infty,0)\cup(0,1)$ with $\lim\limits_{\nu\to 1^{-}} \gamma(\nu)=0=\gamma(1)$ and $\lim\limits_{x\to -\infty} \frac{\gamma(x)}{\sqrt{|x|}}=\tfrac14$. We also get that   $\gamma(x)\II(x)$ blows up as $\nu \to 0^-$ or $0^+$.

 Differentiating  (\ref{rhonu}) and using the definition (\ref{defH}) lead to 
\begin{align}\label{stuff}
 \II(x)\gamma '(x)= \frac{1}{8}+ \II'(x)\gamma (x), \qquad \frac{\gamma ''(x)}{\gamma(x)}= \frac{\II''(x)}{\II(x)}=\frac{1}{4x(1-x)},
  \end{align}
 for $x\in(-\infty,0)\cup(0,1)$. 
 We have $\II'(x)=-\frac{K(x)}{2}<0$ for $x<1$ and $\II(0)=0$. Thus from the second identity we get that $\gamma'(\nu)$ is strictly decreasing in $(-\infty,0)$ and strictly  increasing in $(0,1)$.  From the asymptotics (\ref{limits}) of Proposition \ref{prop:ldpasympt} it is not hard to check that $\lim\limits_{\nu\to -\infty} \gamma'(\nu)=0$ and $\lim\limits_{\nu\to1^{-}} \gamma'(\nu)=-\tfrac18$. This, together with the previous statement,  proves that $\gamma(\nu)$ is decreasing on $(-\infty,0)$ and also on $(0,1]$.

Since $(\gamma(x) \II(x)^{-1})'=\tfrac{1}{8}\II(x)^{-2}$,  L'Hospital's rule  gives
\begin{align}
\lim\limits_{\nu\to 0} \gamma(\nu)=\lim\limits_{\nu\to 0}\frac{\gamma(\nu) \II(\nu)^{-1}}{\II(\nu)^{-1}}=-\tfrac{1}{8}H'(0)=\frac{1}{2\pi}=\gamma(0).\notag
\end{align}
 Then from (\ref{stuff}) it follows that 
 \begin{align*}
 \lim\limits_{\nu\to 0} \gamma''(\nu)\nu=\frac{1}{8\pi}, \qquad \lim\limits_{\nu\to 0} \frac{\gamma'(\nu)}{\log|\nu|}=\frac{1}{8\pi},
 \end{align*}
 and also that 
 \begin{align}\label{yyy}
 \lim\limits_{x\to 0} \frac{\gamma^{(-1)}(\tfrac1{2\pi}+x)}{8\pi \tfrac{x}{\log|x|}}=1.
 \end{align}
This shows that $\gamma(\nu)$ is continuous and strictly decreasing on $(-\infty,1]$. We have $\gamma^{(-1)}(1)=0$ and $\sineldp(0)=\tfrac{1}{64}$.

Note, the fact that $\gamma(x)$ solves $4x(1-x) \gamma''(x)=\gamma(x)$ on $(-\infty,0)\cup(0,1)$ and has the proven asymptotics at $-\infty, 0$ and $1$ uniquely identifies it. The equation $4x(1-x) y''(x)=y(x)$ has two linearly independent solutions on both $(-\infty,0)$ and $(0,1)$.  The function $\II(x)$ also solves the  equation (on both intervals), but with $\II(0)=0$, $\lim\limits_{x\to1^{-}}\II(x)=-1$ and $\lim\limits_{x\to -\infty} \frac{\II(x)}{\sqrt{|x|} \log|x|}=\tfrac12$. This shows that any solution on $(-\infty,0)$ or $(0,1)$ can be expressed as $c_1 \gamma+c_2 \II$ with some constants $c_1, c_2$, and the values of the constants are determined  by the behavior of the solution at the end of the interval.

Using (\ref{stuff}) together with (\ref{defsineldp}) we can also compute that
\begin{align}\notag
\sineldp'(\rho)&= \frac{1}{8} \left[ \frac{1}{8\gamma'(\nu)}+\II(\nu)+ \frac{\gamma (\nu)\II'(\nu)}{\gamma '(\nu)}\right]= \frac{1}{4} \II(\nu), \textup{ and}\\
\sineldp''(\rho)&= \frac{1}{4} \frac{\II'(\nu)}{\gamma '(\nu)}= - \frac{1}{8} \frac{K(\nu)}{\gamma '(\nu)},\label{www}
\notag
\end{align}
where $\nu$ is short  for $\gamma^{-1}(\rho)$. From this $\sineldp(\tfrac1{2\pi})=0$ follows, together with  $\sineldp(x)>0$ for $x\neq \tfrac1{2\pi}$. The asymptotics of $\sineldp(\tfrac1{2\pi}+x)$ as  $x\to 0$ can be obtained from  the definition (\ref{defsineldp}), the asymptotics  (\ref{yyy}), and the fact that $\II(0)=0, \II'(0)=-\tfrac{\pi}{4}$.

Lastly we can look at the asymptotics of $\sineldp(\rho)$ as $\rho \to \infty$.  Recalling again (\ref{limits}) we get   
\begin{align*}
\II(-x) \sim \tfrac{1}{2} \sqrt {x} \log x, \qquad \gamma(-x)&\sim \frac{\sqrt{x}}{4}, \qquad \gamma^{(-1)}(x)\sim 16 x^2\qquad  \textup{ as $x\to \infty$},
\end{align*}
from which $\sineldp (\rho)\sim  \tfrac{1}{2} \rho^2 \log \rho$ follows for $\rho \to \infty$. 
\end{proof}

$ $\\[0pt]
\noindent\textbf{Acknowledgements.} B. Valk\'o was partially supported by the National Science Foundation CAREER award DMS-1053280. The authors would like to thank P.~J.~Forrester for helpful comments and additional references.

\def\cprime{$'$}


\begin{thebibliography}{10}

\bibitem{AbSt}
M.~{Abramowitz} and I.~A. {Stegun}, editors.
\newblock {\em {Handbook of mathematical functions with formulas, graphs, and
  mathematical tables.}}
\newblock Wiley, 1984.

\bibitem{AGZ}
G.~Anderson, A.~Guionnet, and O.~Zeitouni.
\newblock {\em Introduction to random matrices}.
\newblock Cambridge University Press, 2009.

\bibitem{BTW}
E.~L. Basor, C.~A. Tracy, and H.~Widom.
\newblock Asymptotics of level-spacing distributions for random matrices.
\newblock {\em Phys. Rev. Lett.}, 69(1):5--8, 1992.

\bibitem{BAG}
G.~Ben~Arous and A.~Guionnet.
\newblock Large deviations for {W}igner's law and {V}oiculescu's
  non-commutative entropy.
\newblock {\em Probab. Theory Related Fields}, 108(4):517--542, 1997.

\bibitem{variational}
G.~Buttazzo, M.~Giaquinta, and S.~Hildebrant.
\newblock {\em One-dimensional Variational Problems}.
\newblock Oxford lecture series in mathematics and its applications. Clarendon
  Press, Oxford, 1998.

\bibitem{gustafson1}
B.~C. Carlson and J.~L. Gustafson.
\newblock Asymptotic expansion of the first elliptic integral.
\newblock {\em SIAM J. Math Anal.}, 16(5):1072--1092, September 1985.

\bibitem{DIKZ07}
P.~Deift, A.~Its, I.~Krasovsky, and X.~Zhou.
\newblock The {W}idom-{D}yson constant for the gap probability in random matrix
  theory.
\newblock {\em J. Comput. Appl. Math.}, 202(1):26--47, 2007.

\bibitem{DIZ96}
P.~Deift, A.~Its, and X.~Zhou.
\newblock A {Riemann-Hilbert} approach to asymptotic problems arising in the
  theory of random matrices and also in the theory of integrable statistical
  mechanics.
\newblock {\em Ann. Math.}, 146:149--235, 1997.

\bibitem{DemboZeitouni}
A.~Dembo and O.~Zeitouni.
\newblock {\em Large deviations techniques and applications}, volume~38.
\newblock Springer, 2010.

\bibitem{Dyson95}
F.~J. Dyson.
\newblock The {C}oulomb fluid and the fifth {P}ainlev{\'e} transcendent.
\newblock In {\em Chen {N}ing {Y}ang}, pages 131--146. Int. Press, Cambridge,
  MA, 1995.

\bibitem{Ehr06}
T.~Ehrhardt.
\newblock Dyson's constant in the asymptotics of the {F}redholm determinant of
  the sine kernel.
\newblock {\em Comm. Math. Phys.}, 262(2):317--341, 2006.

\bibitem{FS95}
M.~M. Fogler and B.~I. Shklovskii.
\newblock Probability of an eigenvalue number fluctuation in an interval of a
  random matrix spectrum.
\newblock {\em Phys. Rev. Lett.}, 74(17):3312--3315, Apr 1995.

\bibitem{ForBook}
P.~J. Forrester.
\newblock {\em Log-gases and random matrices}, volume~34 of {\em London
  Mathematical Society Monographs Series}.
\newblock Princeton University Press, Princeton, NJ, 2010.

\bibitem{FW}
P.~J. Forrester and N.~S. Witte.
\newblock Asymptotic forms for hard and soft edge general {\^i}² conditional
  gap probabilities.
\newblock {\em Nuclear Physics B}, 859(3):321 -- 340, 2012.

\bibitem{gustafson2}
J.~L. Gustafson.
\newblock {\em Asymptotic formulas for elliptic integrals}.
\newblock PhD thesis, Iowa State University, Ames, Iowa, 1982.

\bibitem{KaratzasRuf2013}
I.~Karatzas and J.~Ruf.
\newblock Distribution of the time to explosion for one-dimensional diffusions.
\newblock \texttt{http://arxiv.org/abs/1303.5899}, 2013.

\bibitem{KarShr}
I.~Karatzas and S.~E. Shreve.
\newblock {\em Brownian motion and stochastic calculus}, volume 113 of {\em
  Graduate Texts in Mathematics}.
\newblock Springer-Verlag, New York, second edition, 1991.

\bibitem{Kr04}
I.~V. Krasovsky.
\newblock Gap probability in the spectrum of random matrices and asymptotics of
  polynomials orthogonal on an arc of the unit circle.
\newblock {\em Int. Math. Res. Not.}, (25):1249--1272, 2004.

\bibitem{KVV}
E.~Kritchevski, B.~Valk{{\'o}}, and B.~Vir{{\'a}}g.
\newblock The scaling limit of the critical one-dimensional random
  {S}chr{\"o}dinger operator.
\newblock {\em Comm. Math. Phys.}, 314(3):775--806, 2012.

\bibitem{mehta}
M.~L. Mehta.
\newblock {\em Random matrices}, volume 142 of {\em Pure and Applied
  Mathematics (Amsterdam)}.
\newblock Elsevier/Academic Press, Amsterdam, third edition, 2004.

\bibitem{Orlicz}
M.~M. Rao and Z.~D. Ren.
\newblock {\em Theory of {O}rlicz spaces}, volume 146 of {\em Monographs and
  Textbooks in Pure and Applied Mathematics}.
\newblock Marcel Dekker Inc., New York, 1991.

\bibitem{TW93}
C.~A. Tracy and H.~Widom.
\newblock Introduction to random matrices.
\newblock In {\em Geometric and quantum aspects of integrable systems
  ({S}cheveningen, 1992)}, volume 424 of {\em Lecture Notes in Phys.}, pages
  103--130. Springer, Berlin, 1993.

\bibitem{BVBV}
B.~Valk\'o and B.~Vir\'ag.
\newblock Continuum limits of random matrices and the {B}rownian carousel.
\newblock {\em Inventiones Math.}, 177:463--508, 2009.

\bibitem{BVBV2}
B.~Valk{{\'o}} and B.~Vir{{\'a}}g.
\newblock Large gaps between random eigenvalues.
\newblock {\em Ann. Probab.}, 38(3):1263--1279, 2010.

\bibitem{Wi96}
H.~Widom.
\newblock The asymptotics of a continuous analogue of orthogonal polynomials.
\newblock {\em J. Approx. Theory}, 77:51--64, 1996.

\end{thebibliography}
\end{document}